\documentclass[11pt,reqno,a4paper]{amsart}
\usepackage{enumerate,cancel,hyperref,mathtools}
\newtheorem{thm}{Theorem}[section]
\newtheorem{fact}[thm]{Fact}
\newtheorem{prop}[thm]{Proposition}
\newtheorem{lemma}[thm]{Lemma}
\newtheorem{cor}[thm]{Corollary}
\newtheorem{question}{Question}
\newtheorem*{thm*}{Theorem}
\theoremstyle{definition}
\newtheorem{defin}[thm]{Definition}

\theoremstyle{remark}
\newtheorem{remark}[thm]{Remark}

\newcommand{\setsep}{:\;}
\newcommand{\floor}[1]{\lfloor #1 \rfloor}
\newcommand{\Lip}{\operatorname{Lip}}
\newcommand{\Span}{\operatorname{span}}
\newcommand{\Id}{\operatorname{Id}}

\newcommand{\Ext}{\operatorname{\operatorname{Ext}^{pt}_0}}
\newcommand{\Exto}{\operatorname{\operatorname{Ext}_0}}
\newcommand{\Rea}{\mathbb{R}}
\newcommand{\Nat}{\mathbb{N}}
\newcommand{\zet}{\mathbb{Z}}
\newcommand{\U}{\mathcal{U}}

\newcommand{\F}{\mathcal{F}}
\newcommand{\ex}{\mathbb{X}}

\newcommand{\Net}{\mathcal{N}}
\newcommand{\complemented}{\stackrel{c}{\hookrightarrow}}

\begin{document}
\title[Isomorphic spaces of Lipschitz functions]{Isomorphisms between spaces of Lipschitz functions}

\author[L. Candido]{Leandro Candido}
\address{Universidade Federal de S\~ao Paulo - UNIFESP. Instituto de Ci\^encia e Tecnologia. Departamento de Matem\'atica. S\~ao Jos\'e dos Campos - SP, Brasil}
\email{leandro.candido@unifesp.br}
\author[M. C\'uth]{Marek C\'uth}
\address{Faculty of Mathematics and Physics, Department of Mathematical Analysis\\
Charles University\\ 
186 75 Praha 8\\
Czech Republic}
\email{cuth@karlin.mff.cuni.cz}
\author[M. Doucha]{Michal Doucha}
\address{Institute of Mathematics\\
 Czech Academy of Sciences\\
115 67 Praha 1\\
Czech Republic}
\email{doucha@math.cas.cz}


\thanks{L. Candido was supported by FAPESP No. 2016/25574-8. M. C\'uth has been supported by Charles University Research program No. UNCE/SCI/023 and by the Research grant GA\v{C}R 17-04197Y. M. Doucha was supported by the GA\v CR project 16-34860L and RVO: 67985840.}


\begin{abstract}We develop tools for proving isomorphisms of normed spaces of Lipschitz functions over various doubling metric spaces and Banach spaces. In particular, we show that $\Lip_0(\zet^d)\simeq\Lip_0(\Rea^d)$, for all $d\in\Nat$. More generally, we e.g. show that $\Lip_0(\Gamma)\simeq \Lip_0(G)$, where $\Gamma$ is from a large class of finitely generated nilpotent groups and $G$ is its Mal'cev closure; or that $\Lip_0(\ell_p)\simeq\Lip_0(L_p)$, for all $1\leq p<\infty$.

We leave a large area for further possible research.
\end{abstract}
\maketitle

\section*{Introduction}
The goal of the paper is to develop and apply tools that in some situations allow us to identify, up to isomorphism, two normed spaces of Lipschitz functions over two different metric spaces - provided that these metric spaces share a similar geometry when one looks only at their finite subsets. A canonical result of this type, which stimulated our more general considerations, is that the spaces of Lipschitz functions vanishing at zero over $\zet^d$ and $\Rea^d$ are isomorphic, for any $d\in\Nat$. The tools that we develop along the way, proving that result, have much greater applicability. We use them to investigate spaces of Lipschitz functions over spaces coming both from finite-dimensional (sub-)Riemannian geometry, such as various Lie groups, and from infinite-dimensional functional analysis, such as various Banach spaces.
\subsection*{Spaces of Lipschitz functions.}
Let $M$ be a metric space, often considered with a distinguished point that is usually denoted by $0$, and consider the set of all real-valued Lipschitz functions on $M$, denoted by $\Lip(M)$, and the set of all real-valued Lipschitz functions on $M$ that vanish at $0$, denoted by $\Lip_0(M)$. These spaces have an obvious vector space structure and one can equip them with the (pseudo-)norm $\|\cdot\|_{\Lip}$ of `smallest Lipschitz constant', i.e. $$\|f\|_{\Lip}=\sup_{x\neq y\in M} \frac{|f(x)-f(y)|}{d(x,y)},$$ for $f\in\Lip(M)$. The space $(\Lip_0(M),\|\cdot\|_{\Lip})$ then becomes a Banach space, while $\Lip(M)$, because of the constant functions on which $\|\cdot\|_{\Lip}$ vanishes, is more often considered with the norm $\|\cdot\|_{\Lip}+|f(0)|$; then it becomes a Banach space as well.

There is an apparent relation of these spaces to spaces of differentiable functions on $\Rea^n$ as $\|f'\|_\infty$ is nothing but $\|f\|_{\Lip}$. However in our setting we can work with metric spaces that have no obvious differentiable structure. Let us note that spaces of Lipschitz functions are widely studied in applied mathematics, they naturally appear in many areas of computer science, including for example in machine learning (see e.g. \cite{vLBou}, or other work of these authors). We also note that the spaces $\Lip(M)$ and $\Lip_0(M)$ are also studied when equipped with a richer structure on them than just Banach space structure. They have the lattice structure (although they are not Banach lattices since the Lipschitz norm is not increasing on the positive cone), and when $M$ is bounded, they are also topological algebras (they \emph{almost} are Banach algebras with their norm and are sometimes called \emph{weak Banach algebras}); we refer the reader to the monograph \cite{Weaver}.

In Banach space theory, $\Lip_0(M)$ has been more studied recently because of its canonical predual space, the \emph{Lipschitz-free space} over $M$ (also called the \emph{Arens-Eells space}), denoted further by $\F(M)$. This is a Banach space together with an isometric embedding, preserving $0$, $\delta:M\rightarrow \F(M)$, characterized by the universal property that for any Lipschitz function $f:M\rightarrow X$ to a Banach space preserving $0$ there is a unique extension $F:\F(M)\rightarrow X$ with $\|F\|=\|f\|_{\Lip}$ satisfying $f=F\circ\delta$. This property also allows one to isometrically identify $\Lip_0(M)$ with the dual Banach space $\F(M)^*$ - a fact that we will use.
\medskip

Let us provide some sample theorems that we prove here.
\begin{thm*}
For any $d\in\Nat$, we have $\Lip_0(\zet^d)\simeq \Lip_0(\Rea^d)$.

More generally, let $\Gamma$ be a finitely generated nilpotent torsion-free group equipped with its word metric and let $G$ be its Mal'cev closure, the simply connected Lie group containing $\Gamma$ as a discrete subgroup which is also a net. Suppose that $G$ is Carnot and it is equipped with its Carnot-Carath\' edory distance. Then $\Lip_0(\Gamma)\simeq \Lip_0(G)$
\end{thm*}
Carnot groups appearing in the previous theorem can be seen as a non-commutative generalization of finite-dimensional Banach spaces. In particular, they admit dilations for any positive factor $\lambda$, they are doubling, homogeneous and geodesic - geometric properties that we use in our proofs. It turns out that these properties characterize Carnot groups among metric spaces (\cite{lD}).

Besides these finite-dimensional spaces which enjoy the doubling condition, essential in many arguments, we can still derive many results when the metric spaces in question are Banach spaces. Some sample results are as follows.

\begin{thm*}
For all $1\leq p<\infty$, we have $\Lip_0(\ell_p)\simeq \Lip_0(L_p)$.
Moreover,
\begin{itemize}
\item Let $X$ be a separable infinite dimensional $\mathcal{L}_p$-space for some $p\in [1,\infty]$. Then $\Lip_0(X)\simeq \Lip_0(\ell_p)$ if $p<\infty$, and $\Lip_0(X)\simeq \Lip_0(c_0)$ if $p=\infty$.
\item If $1<p<\infty$, then there is a sequence $(X_n)_{n=1}^\infty$ of Banach spaces, each of which isomorphic to 
$\ell_2$, such that $\Lip_0(L_p)\simeq\bigoplus_{\ell_\infty} \Lip_0(X_n).$
\item For a Young function $\Phi$ satisfying the $\Delta_2$-condition and any sequence $(\alpha_n)_{n=1}^{\infty}$ of positive reals converging to $0$, $\Lip_0(L_\Phi(\Rea))\simeq \bigoplus_{\ell_\infty} \Lip_0(\ell_{\alpha_n\Phi})$, where $\ell_{\alpha_n\Phi}$ is the Orlicz sequence space associated to the function $\alpha_n\Phi$.
\end{itemize}
\end{thm*}

\subsection*{Preliminaries and notations}

The notation and terminology we use are relatively standard. If $X$ is a topological space and $A\subset X$, $\overline{A}$ stands for the closure of $A$. If $(M,d)$ is a metric space, $x\in M$ and $r\geq 0$, we use $B(x,r)$ to denote the closed ball $\{y\in M\setsep d(x,y)\leq r\}$. A metric space $(M,d)$ is \emph{uniformly discrete} if it is $\varepsilon$-separated for some $\varepsilon > 0$, that is, $d(x,y)\geq \varepsilon$ for every $x,y\in M$, $x\neq y$. A subset $N$ of a metric space $(M,d)$ is \emph{$\varepsilon$-dense}, if for each $x\in M$ there is $y\in N$ with $d(x,y)<\varepsilon$. A \emph{net} in a metric space is a subset which is $\varepsilon$-separated and $\delta$-dense in $M$ for some $\varepsilon, \delta > 0$. If $X$ and $Y$ are Banach spaces, the symbol $Y\complemented X$ means that $Y$ is linearly isomorphic to a complemented subspace of $X$. The symbol $Y\simeq X$ means that $Y$ is linearly isomorphic to $X$.

Our results concerning isomorphisms between spaces of Lipschitz functions are always obtained using the following special case of the standard Pe\l czy\'nski decomposition method proved originally in \cite{P60}. We state it here for the convenience of the reader.
\begin{thm}\label{t:pel}Let $X$, $Y$ be Banach spaces with $X\simeq\bigoplus_{\ell_\infty} X$, $Y\complemented X$ and $X\complemented Y$. Then $X\simeq Y$.
\end{thm}

In order to satisfy the assumption ``$X\simeq\bigoplus_{\ell_\infty} X$'' from Theorem~\ref{t:pel}, we use sometimes the following result by Kaufmann, see \cite[Theorem 3.1]{K15}.
\begin{thm}\label{thm:kaufman}Let $X$ be a Banach space. Then $\F(X)\simeq \bigoplus_{\ell_1}\F(X)$. In particular, $\Lip_0(X)\simeq \bigoplus_{\ell_\infty} \Lip_0(X)$.
\end{thm}
\noindent Note that the ``In particular'' part of Theorem~\ref{thm:kaufman} is further generalized in our Theorem~\ref{thm:isoToEllInftySum}.

Finally, we use the following two basic strategies in order to find a complemented copy of a $\Lip_0(N)$ space in $\Lip_0(M)$. The first one is a standard way of constructing projections, which we formulate here for the purpose of a further reference. The proof is easy, standard and so we omit it.

\begin{fact}\label{fact:findingProjection}Let $X$, $Y$ be Banach spaces and $S:X\to Y$, $T:Y\to X$ be bounded linear operators with $S\circ T = \Id$. Then $Y\complemented X$. Moreover, $T$ is an isomorphism, $T\circ S$ is a projection onto $T(Y)$ and $Y$ is $\|S\|\|T\|$-isomorphic to a $\|S\|\|T\|$-complemented subspace of $X$.
\end{fact}

The second strategy is based on the existence of linear extension operators. Given a pointed metric space $(M, d, 0)$ and a subset $F$ containing 0, let $\Exto(F,M)$ denote the set of all extensions $E:\Lip_0(F)\to \Lip_0(M)$ which are linear and continuous. By $\Ext(F,M)$ we denote the set of those $E\in \Exto(F,M)$ which are pointwise-to-pointwise continuous. The following is a classical fact (for the proof concerning Lipschitz-free space one may consult e.g. \cite[Section 2.1]{K15}).

\begin{fact}\label{fact:extensions}Let $(M, d, 0)$ be a pointed metric space and let $N\subset M$ be a subset with $0\in N$ and let $T\in \Exto(N,M)$. Then $\Lip_0(N)$ is isometric to a $\|T\|$-complemented subspace of $\Lip_0(M)$. Moreover, if $T\in \Ext(N,M)$, then $\F(N)$ is isometric to a $\|T\|$-complemented subspace of $\F(M)$.
\end{fact}

Let us explicitly mention one special case of the above.
\begin{fact}\label{fact:retractionImpliesComplementedLipSpaces}Let $M$ be a metric space, $N$ be a subset of $M$ and let $R:M\to N$ be a Lipschitz retraction.
Then $\Lip_0(N)$ is isometric to a $\|R\|_{Lip}$-complemented subspace of $\Lip_0(M)$.
\end{fact}
\begin{proof}Consider the mapping $E(f):=f\circ R$, $f\in\Lip_0(N)$. Then $E\in\Exto(N,M)$ and we may apply Fact~\ref{fact:extensions}.
\end{proof}

The structure of the paper is following. In the next section, we develop the technical tools needed for the proofs of our main results, and in particular derive that $\Lip_0(\zet^d)\simeq \Lip_0(\Rea^d)$, for all $d\in \Nat$, see Theorem~\ref{cor:RAndZIso}. In the second section, we put that last result into the wider context of Lipschitz spaces on finitely generated groups and Lipschitz spaces on Lie groups. The third section contains applications to Lipschitz spaces on infinite-dimensional Banach spaces. The last section contains several open questions.

\section{Key Lemma and isomorphisms to \texorpdfstring{$\ell_\infty$}--sums}\label{sec:principles}

In this section we summarize the basic principles which lead to the proof of Theorem~\ref{thm:TheIsomorphismTheorem} from which it follows that $\Lip_0(\Rea^d)\simeq \Lip_0(\zet^d)$ for every $d\in\Nat$. Those principles will be used in subsequent sections in various other situations.

\subsection{Key Lemma and its first consequences}

It seems to us that the key result of the current paper which enables us to prove statements concerning isomorphisms between $\Lip_0(M)$ spaces which do not pass to their preduals is Lemma~\ref{lem:key}, which we therefore call the Key Lemma. In order to enlighten its statement, let us start with some definitions.

\begin{defin}\label{defin:rescaling}
 If $(M,d,0)$ is a pointed metric space and $\lambda>0$ is a constant, a \emph{$\lambda$-rescaling} of $M$, denoted by $\lambda M$, is the pointed metric space $(M,\tfrac{d}{\lambda},0)$; that is, we have $\tfrac{d}{\lambda}(x,y)=\frac{d(x,y)}{\lambda}$ for $x,y\in M$. 
\end{defin}

Let us note that the above definition is chosen in such a way that whenever $X$ is a normed linear space, then the mapping $T_\lambda:X\to \lambda X$ defined as $T_\lambda(x):=\lambda x$, $x\in X$ is an isometry.

\begin{defin}
We call a pair of pointed metric spaces $(X,\ex)$ a \emph{good pair} if there exists a sequence of positive real numbers $(r_n)_{n=1}^{\infty}$ such that
\begin{itemize}
\item There is a sequence of bi-Lipschitz embeddings $i_n:r_nX\to \ex$ with $\sup_n\{\Lip(i_n),\Lip(i_n^{-1})\} < \infty$ and $i_n(0)=0$ for each $n\in\Nat$.
\item For each $x\in\ex$ there is a sequence $(x_n)_{n=1}^\infty$ in $X$ such that $i_n(x_n)\to x$.
\end{itemize}
\end{defin}

An important example of a good pair is $(\zet^d,\Rea^d)$. We will see other examples later.

Let us recall that, given $C,D>0$ and metric spaces $(M,\rho)$ and $(N,\sigma)$, a mapping $f:M\to N$ is \emph{$(D,C)$-Coarse Lipschitz} if for each $x,y\in M$ we have  $\sigma(f(x),f(y))\leq C + D\rho(x,y)$.

\begin{lemma}[Key Lemma]\label{lem:key}
Let $(M,d)$ be a pointed metric space and $(M_n)_{n=1}^\infty$ be a sequence of pointed metric spaces such that there are $D_1,D_2>0$, a sequence $(\varepsilon_n)$ of non-negative numbers with $\varepsilon_n\to 0$, and a sequence of injective $D_1$-Lipschitz maps $i_n: M_n\rightarrow M$, with $i_n(0)=0$, such that the inverse mappings $i^{-1}_n:i_n(M_n)\to M_n$ are $(D_2,\varepsilon_n)$-Coarse Lipschitz, and finally for every $x\in M$ there is a sequence $(x_n)_{n=1}^\infty$ with $i_n(x_n)\to x$. Then 
\begin{equation}\label{formula}\Lip_0(M)\complemented\bigoplus_{\ell_\infty}\Lip_0(M_n).
\end{equation}
Quantitatively, $\Lip_0(M)$ is $D_1D_2$-isomorphic to a $D_1D_2$-complemented subspace of $\bigoplus_{\ell_\infty}\Lip_0(M_n)$.

In particular, (\ref{formula}) holds if $(M_n)_{n=1}^\infty$ is an increasing sequence of subsets of $M$ such that $\bigcup_{n \in \Nat}M_n$ 
is dense in $M$.
\end{lemma}
\begin{proof}
Let us consider the operator $T:\Lip_0(M)\to \bigoplus_{\ell_\infty} \Lip_0(M_n)$ defined by
\[
	\big(T(f)(n)\big)(x):=f\big(i_n(x)\big),\quad x\in M_n, n\in\Nat.
\]
Clearly, $T$ is linear and $\|T\|\leq D_1$.

Further, fix a non-principal ultrafilter $\U$ on $\Nat$ and consider $S:\bigoplus_{\ell_\infty} \Lip_0(M_n)\to \Lip_0(M)$ defined by
\[
S((f_n))(x):=\lim_\U f_n(x_n),\quad i_n(x_n)\to x, (f_n)\in \bigoplus_{\ell_\infty} \Lip_0(M_n).
\]
Choose arbitrarily sequences $(x_n)_n$ and $(y_n)_n$ such that $i_n(x_n)\to x$ and $i_n(y_n)\to y$. Then
\begin{align*}
|S((f_n))(x) -  S((f_n))(y)|=|\lim_\U f_n(x_n)-\lim_\U f_n(y_n)|=
 \lim_\U |f_n(x_n)-f_n(y_n)|\leq\\ \|(f_n)\| \lim_\U d_{M_n}(x_n,y_n)= \|(f_n)\| \lim_\U d_{M_n}(i^{-1}_n(i_n(x_n)),i^{-1}_n(i_n(y_n)))\leq\\
\leq \|(f_n)\| \big(\lim_\U D_2d(i_n(x_n),i_n(y_n))+\varepsilon_n\big) = D_2 \|(f_n)\|d(x,y).
\end{align*}

Hence, $S$ is well-defined and $\|S\|\leq D_2$. It is straightforward to see that it is also linear. Moreover, we have $S\circ T = \Id$. Indeed, for $f\in \Lip_0(M)$ and $x\in M$ we have
\[
S(T(f))(x) = \lim_{\U}(T(f)(n))(x_n) = \lim_{\U} f\big(i_n(x_n)\big) = f(x),
\]
where $(x_n)_{n=1}^\infty$ is a sequence satisfying $i_n(x_n)\to x$. Hence, by Fact~\ref{fact:findingProjection}, we have $\Lip_0(M)\complemented\bigoplus_{\ell_\infty}\Lip_0(M_n)$.

The ``In particular'' part follows by taking $i_n$ to be the canonical injection for each $n$.
\end{proof}

\begin{remark}We have been informed by the anonymous referee that the proof of Lemma~\ref{lem:key} bears some similarities with the techniques used in \cite{Go15}.\end{remark}

\begin{prop}\label{prop:PropForGoodPairs}For every good pair $(X,\ex)$ we have $\Lip_0(\ex)\complemented \bigoplus_{\ell_\infty} \Lip_0(X)$.

Moreover, if $\Lip_0(X)\complemented \Lip_0(\ex)$ then $\bigoplus_{\ell_\infty} \Lip_0(\ex)\simeq \bigoplus_{\ell_\infty} \Lip_0(X)$.
\end{prop}
\begin{proof}
Fix an appropriate sequence $(r_n)_{n=1}^\infty$. By Lemma~\ref{lem:key}, we have $\Lip_0(\ex)\complemented \bigoplus_{\ell_\infty} \Lip_0(r_nX)$. Since $\Lip_0(X)$ is isometric to $\Lip_0(r_nX)$ for each $n\in\Nat$, we get $\Lip_0(\ex)\complemented \bigoplus_{\ell_\infty} \Lip_0(X)$. If $\Lip_0(X)\complemented \Lip_0(\ex)$, then we have \[\bigoplus_{\ell_\infty} \Lip_0(\ex)\complemented \bigoplus_{\ell_\infty}\bigoplus_{\ell_\infty} \Lip_0(X)\simeq \bigoplus_{\ell_\infty} \Lip_0(X)\complemented \bigoplus_{\ell_\infty} \Lip_0(\ex).\]
Hence, the ``Moreover'' part follows by the Pe\l czy\'nski decomposition method.
\end{proof}

An important class of metric spaces, where the assumption of the ``Moreover'' part of Proposition~\ref{prop:PropForGoodPairs} is satisfied are the doubling metric spaces, see Proposition~\ref{prop:RdIsomorphicWithLinftyZd} below.

\begin{defin}
A metric space $M$ is \emph{doubling} if there exists a constant $D(M)\in\Nat$, called the \emph{doubling constant}, such that every ball of radius $r>0$ in $M$ can be covered by at most $D(M)$-many balls of radius $r/2$.
\end{defin}
Euclidean spaces are typical examples of doubling spaces. However, not every doubling space can be bi-Lipschitz embedded into an Euclidean space, see \cite[Theorem 7.1]{S96}. Also, every subspace of a doubling metric space is again doubling. This is easy and well-known, but we record it here for the convenience of the reader.

\begin{fact}\label{fact:ballsInDoublingspaces}
Let $M$ be a doubling metric space. Then every subspace of $M$ is again doubling with doubling constant bounded by $D(M)^2$.
\end{fact}
\begin{proof}
Let $N$ be a subspace of $M$ and let $B_N(x,r)$ be a ball in $N$. Put $C:=D(M)$. Since $M$ is doubling with constant $C$, there are points $y_1,\ldots,y_{C^2}\in M$ such that $B_M(x,r)\subset \bigcup_{i=1}^{C^2} B_M(y_i,r/4)$. For each $i=1,\ldots,C^2$, pick $z_i\in N\cap B(y_i,r/4)$ if it exists (otherwise put $z_i=x$). Then $B_N(z_i,r/2)\supset B_N(y_i,r/4)$; hence, balls $B_N(z_1,r/2),\ldots B_N(z_{C^2},r/2)$ cover $B_N(x,r)$.
\end{proof}

\begin{prop}\label{prop:complInRd}There exists a universal constant $C>0$ such that for every metric space $M$ and a for every subspace $N\subset M$ which is doubling, there exists $E\in\Ext(N,M)$ with $\|E\|\leq C(1 + \log D(N))$.

Consequently, we have $\F(N)\complemented \F(M)$  and $\Lip_0(N)\complemented \Lip_0(M)$.
\end{prop}
\begin{proof}
By \cite[Lemma 3.8, Corollary 3.12 and Theorem 4.1]{LN05}, $M$ admits a $C(1+\log D(N))$-gentle partition of unity with respect to $N$. It was observed in \cite[page 2325]{LP13} that this implies the existence of $E\in\Ext(N,M)$ with $\|E\|\leq 3C(1+\log D(N))$. The rest follows immediately from Fact~\ref{fact:extensions}.
\end{proof}

\begin{prop}\label{prop:RdIsomorphicWithLinftyZd}For every good pair $(X,\ex)$ where $X$ is doubling we have $\bigoplus_{\ell_\infty} \Lip_0(\ex)\simeq \bigoplus_{\ell_\infty} \Lip_0(X)$.
\end{prop}
\begin{proof}This is a consequence of Proposition~\ref{prop:PropForGoodPairs} and Proposition~\ref{prop:complInRd}.
\end{proof}

Note that there is no hope that Proposition~\ref{prop:RdIsomorphicWithLinftyZd} would pass to the preduals in general, see Fact~\ref{prop:freeOverRdAndZdnotIso}.

\subsection{Isomorphisms to \texorpdfstring{$\ell_\infty$}--sums}

Let us start by recalling the following result by Kalton, see \cite[Lemma 4.1]{K04}.

\begin{prop}\label{prop:kalton}Let $(M,d,0)$ be a pointed metric space and let $(r_k)$ and $(s_k)$ be sequences of numbers with $s_k < r_k < s_{k+1}$ for $k\in\Nat$ and, for each $k\in\Nat$, denote by $C_k$ the set $\{x\in M\setsep 2^{s_k} < d(x,0)\leq 2^{r_k}\}$.

If $\inf\{s_{k+1} - r_{k}\setsep k\in\Nat\} = \theta > 0$, then
\[
	\Lip_0(\{0\}\cup \bigcup_{k=1}^\infty C_k)\simeq \bigoplus_{\ell_\infty}\Lip_0(\{0\}\cup C_k)\quad\text{and}\quad \F(\{0\}\cup \bigcup_{k=1}^\infty C_k)\simeq \bigoplus_{\ell_1} \F(\{0\}\cup C_k).
\]
\end{prop}

\begin{lemma}\label{lem:isomorphicSums}Let $M$ be a pointed metric space and $(M_n)_{n=1}^\infty$ an increasing sequence of subsets of $M$ such that $\bigcup_{n\in\Nat} M_n$ is dense in $M$ and that there exists $C>0$ such that for each $n\in\Nat$ there is $E\in\Exto(M_n,M)$ with $\|E\|\leq C$. Then
\begin{equation}\label{eq:firstStep}
	\bigoplus_{\ell_\infty}\Lip_0(M)\simeq \bigoplus_{\ell_\infty} \Lip_0(M_n).
\end{equation}
\end{lemma}
\begin{proof}
By Lemma~\ref{lem:key}, we have 
\begin{equation}\label{eq:sumInSumSum}
\bigoplus_{\ell_\infty}\Lip_0(M)\complemented \bigoplus_{\ell_\infty}\Big(\bigoplus_{\ell_\infty}\Lip_0(M_n)\Big).
\end{equation}
Fix a bijection $\pi:\Nat^2\to \Nat$ such that $\pi(i,j)\geq i$ for every $(i,j)\in \Nat^2$. Pick $E_i\in\Exto(M_i,M)$ with $\|E_i\|\leq C$ for each $i\in\Nat$ and consider the operator $T:\bigoplus_{\ell_\infty}\Big(\bigoplus_{\ell_\infty}\Lip_0(M_n)\Big)\to \bigoplus_{\ell_\infty}\Lip_0(M_n)$ 
given by
\[\begin{split}
T((f_{i,j}))_{\pi(r,s)}(x):=E_r(f_{r,s})(x),\quad & x\in M_{\pi(r,s)},\; (f_{i,j})\text{ is such that}\\
& f_{i,j}\in \Lip_0\big(M_i\big)\text{ for each $j\in\Nat$}.
\end{split}\]

Then we have $\|T\|\leq C$. Further, consider the operator $S:\bigoplus_{\ell_\infty}\Lip_0(M_n)\to \bigoplus_{\ell_\infty}\Big(\bigoplus_{\ell_\infty}\Lip_0(M_n)\Big)$ 
given by
\[
S((f_n))_{(i,j)}(x):=f_{\pi(i,j)}(x),\quad x\in M_i, (f_n)\in \bigoplus_{\ell_\infty}\Lip_0(M_n).
\]
Then we have $\|S\|\leq 1$ and it is easy to verify that $S\circ T = \Id$; hence, by Fact~\ref{fact:findingProjection}, we have
\[
\bigoplus_{\ell_\infty}\Big(\bigoplus_{\ell_\infty}\Lip_0(M_n)\Big)\complemented \bigoplus_{\ell_\infty}\Lip_0(M_n),
\]
which together with \eqref{eq:sumInSumSum} gives $\bigoplus_{\ell_\infty}\Lip_0(M)\complemented \bigoplus_{\ell_\infty} \Lip_0(M_n)$ . On the other hand, by Fact~\ref{fact:extensions}, each $\Lip_0(M_n)$ is $C$-complemented in $\Lip_0(M)$ and so $\bigoplus_{\ell_\infty} \Lip_0(M_n)\complemented \bigoplus_{\ell_\infty}\Lip_0(M)$; thus, Pe\l czy\'nski decomposition method gives \eqref{eq:firstStep}.
\end{proof}

\begin{defin}
If there exists an isometric embedding $i:M\to \lambda M$ with $i(0)=0$, then we denote $i(x)$ by $\lambda x$ and we say that $M$ is \emph{$\lambda$-closed}. Note that in this case we have $d(\lambda x,\lambda y) = \lambda d(x,y)$ for every $x,y\in M$.

Following \cite{lD}, we shall say a pointed metric space $M$ is \emph{self-similar} supposing that there exists $\lambda >1$ such that there is a bijection $f:M\to M$ with $d(f(x),f(y)) = \lambda d(x,y)$.

Moreover, a metric space $M$ is called \emph{homogeneous} if for every $x,y\in M$ there exists an isometry of $M$ sending $x$ to $y$.
\end{defin}

\begin{thm}\label{thm:isoToEllInftySum}Let $(M,d)$ be an unbounded pointed metric space satisfying the following two conditions:
\begin{enumerate}[(i)]
\item\label{cond:extensionFromBalls} there is $C>0$ such that for every $r>0$ there is $E\in\Exto(B(0,r),M)$ with $\|E\|\leq C$;
\item\label{cond:extensionFromCoronas} whenever $C_k$ are as in Proposition~\ref{prop:kalton}, we have $\Lip_0(\{0\}\cup \bigcup_{k=1}^\infty C_k)\complemented \Lip_0(M)$.
\end{enumerate}
Moreover, let $M$ satisfy one of the following two conditions
\begin{enumerate}[(a)]
\item $M$ is homogeneous;
\item $M$ is uniformly discrete, $\lambda$-closed for some $\lambda > 1$ and for each $n\in\Nat$ and $r>0$ there is $E\in\Exto(\lambda^n B(0,r),M)$ with $\|E\|\leq C$ where we consider metric $d$ on the set $\lambda^n B(0,r) = \{\lambda^n x\setsep x\in B(0,r)\}$.
\end{enumerate}

Then
\[
\Lip_0(M)\simeq \bigoplus_{\ell_\infty} \Lip_0(M).
\]
\end{thm}
\begin{proof}Pick an increasing sequence $(r_n)_{n=1}^\infty$ of positive numbers such that $r_n\to\infty$. If $M$ is not homogeneous, it is $\varepsilon$-separated and $\lambda$-closed for some $\lambda > 1$; in this case we pick $(r_n)$ in such a way that moreover we have $\varepsilon < r_n\leq \tfrac{\varepsilon}{4}\lambda^{2^n}$, $n\in\Nat$.
By Lemma~\ref{lem:isomorphicSums} we have $\Lip_0(M)\complemented \bigoplus_{\ell_\infty}\Lip_0(M)\complemented \bigoplus_{\ell_\infty} \Lip_0\big(B(0,r_n)\big)$. Hence, by the Pe\l czy\'nski decomposition method, it suffices to show $\bigoplus_{\ell_\infty} \Lip_0\big(B(0,r_n)\big) \complemented \Lip_0(M)$. Now we distinguish the two cases.

\smallskip
\noindent \textbf{Case 1 ($M$ is homogeneous):} Put $r_0=0$ and $x_0 = 0$. We inductively find a sequence $(x_n)_{n=1}^\infty$ such that $d(x_n,0) > 3r_n + 2d(x_{n-1},0)$ for each $n\in\Nat$. For each $n\in\Nat$ we put
\[
C_n := \{x\in M\setsep 2(r_{n-1} + d(x_{n-1},0)) < d(x,0)\leq d(x_n,0) + r_n\}.
\]
Now we easily verify that $B(x_n,r_n)\subset C_n$. Moreover, the assumptions of Proposition~\ref{prop:kalton} are satisfied for sets $C_n$ with $\theta = 1$ as the following computation shows:
\[
	\log_2\big(2(r_n + d(x_n,0))\big) - \log_2(d(x_n,0) + r_n) = \log_2 2 = 1.
\]
Clearly, $B(x_n,r_n)$ is isometric to $B(0,r_n)$ since $M$ is homogeneous; hence, by the condition \eqref{cond:extensionFromBalls}, we have that $\Lip_0(B(0,r_n))\simeq \Lip_0(B(x_n,r_n))$ is isometric to a $C$-complemented subspace of $\Lip_0(\{0\}\cup C_n)$.

\smallskip
\noindent \textbf{Case 2 ($M$ is $\varepsilon$-separated and $\lambda$-closed):} Note that it follows from the definition that $M$ is $\lambda^n$-closed for each $n\in\Nat$. We put
\[
C_n := \{x\in M\setsep \tfrac{\varepsilon}{2}\lambda^{2^n} < d(x,0)\leq r_n \lambda^{2^n}\},\;n\in\Nat.
\]
Now the mapping $x\mapsto \lambda^{2^n} x$ maps $B(0,r_n)$ into $C_n$. Moreover, the assumptions of Proposition~\ref{prop:kalton} are satisfied for sets $C_n$ with $\theta = 1$ as the following computation shows:
\[
	\log_2(\tfrac{\varepsilon}{2}\lambda^{2^{n+1}}) - \log_2(r_n \lambda^{2^n})  \geq  \log_2(2)  = 1.
\]
Finally, let us observe that $\Lip_0(B(0,r_n))$ is isometric to a $C$-complemented subspace of $\Lip_0(\{0\}\cup C_n)$. Indeed, the mapping $f\mapsto \lambda^{2^n} \big(f\circ (\lambda^{2^n})^{-1}\big)$ is isometry between $\Lip_0(B(0,r_n))$ and $\Lip_0((\lambda^{2^n}B(0,r_n),d))$, where the later is isometric to a $C$-complemented subspace of $\Lip_0(\{0\}\cup C_n)$ by the assumption.

\smallskip
Hence, in both cases, 
\[
\bigoplus_{\ell_\infty}\Lip_0 \big(B(0,r_n)\big)\complemented \bigoplus_{\ell_\infty}\Lip_0 \big(\{0\}\cup C_n\big)\simeq\Lip_0\Big(\{0\}\cup \bigcup_{n=1}^\infty C_n\Big)\complemented \Lip_0(M),
\]
where the second and third follow from Proposition~\ref{prop:kalton} and condition \eqref{cond:extensionFromCoronas} respectively. Therefore, we have $\bigoplus_{\ell_\infty} \Lip_0\big(B(0,r_n)\big) \complemented \Lip_0(M)$, which is what we wanted to show.
\end{proof}

\begin{remark}
The assumptions \eqref{cond:extensionFromBalls} and \eqref{cond:extensionFromCoronas} of Theorem~\ref{thm:isoToEllInftySum} are satisfied in Banach spaces as observed e.g. in the proof of \cite[Theorem 3.1]{K15} or in doubling metric spaces which follows from Proposition~\ref{prop:complInRd} and Fact~\ref{fact:ballsInDoublingspaces}.

It follows from Lemma~\ref{lem:isomorphicSums} that under the assumptions of Theorem~\ref{thm:isoToEllInftySum}, $\Lip_0(M)$ is isomorphic to the $\ell_\infty$-sum of Lipschitz spaces over its balls as well. Moreover, from the proof of Theorem~\ref{thm:isoToEllInftySum} it follows that $\Lip_0(M)$ is isomorphic to the $\ell_\infty$-sum of Lipschitz spaces over its annuli.
\end{remark}

\begin{thm}\label{thm:TheIsomorphismTheorem}
Let $\ex$ be a pointed homogeneous doubling metric space. Let $X$ be a subset of $\ex$ such that $(X,\ex)$ form a good pair and such that
\begin{itemize}
\item either $X$ is uniformly discrete and $\lambda$-closed for some $\lambda > 1$,
\item or $X$ is homogeneous and unbounded.
\end{itemize}
Then we have $\Lip_0(\ex)\simeq \Lip_0(X)$.
\end{thm}
\begin{proof}By Proposition~\ref{prop:RdIsomorphicWithLinftyZd}, it suffices to show we have $\Lip_0(X)\simeq \bigoplus_{\ell_\infty}\Lip_0(X)$ and $\Lip_0(\ex)\simeq \bigoplus_{\ell_\infty}\Lip_0(\ex)$. However, this follows from Theorem~\ref{thm:isoToEllInftySum}, because the assumptions \eqref{cond:extensionFromBalls} and \eqref{cond:extensionFromCoronas} are satisfied by Proposition~\ref{prop:complInRd} and Fact~\ref{fact:ballsInDoublingspaces}.
\end{proof}

The following theorem is our main application in this section and it follows immediately from Theorem \ref{thm:TheIsomorphismTheorem}. Let us note that this result is further generalized and commented in the next section, see Theorem~\ref{thm:LipofCarnot} and the discussion below it.
\begin{thm}\label{cor:RAndZIso}
For every $d\in\Nat$, we have $\Lip_0(\zet^d)\simeq \Lip_0(\Rea^d)$.
\end{thm}

Let us close this section with few observations which will be used later.

\begin{lemma}\label{lem:netsInSSspaces}
Let $\ex$ be a pointed self-similar space and let $X$ be an $\varepsilon$-dense subset, for some $\varepsilon>0$, of $\ex$  containing $0$, the distinguished point. Then $(X,\ex)$ is a good pair.
\end{lemma}
\begin{proof}
Let $\lambda>1$ so that there exists a bijection $f:\ex\to \ex$ with $d(f(x),f(y)) = \lambda d(x,y)$. We will denote the inverse of $f$ by 
$g$ and by $f^n$ and $g^n$ the composition of $n$ functions $f$ and $g$ respectively. 
Since $X$ is $\varepsilon$-dense, for every $x\in\ex$ there is $x'\in X$ with $d(x,x')<\varepsilon$. It follows that for every $n\in\Nat$ and every $x\in\ex$ there is $x'\in X$ with $d(g^n(x),x')<\varepsilon$; hence, 
\[d(x,f^n(x'))=\lambda^{-n}d(g^n(x),g^n(f^n(x')))=\lambda^{-n}d(g^n(x),x')\leq \lambda^{-n}\epsilon.\]
This implies that the mappings $X\ni x\mapsto f^n(x)\in \ex$ are isometric embeddings of $\lambda^{-n} X = (X,\lambda^n\cdot d,0)$ into $\ex$ and, for each $x\in\ex$, there is a sequence $(x_n)_{n=1}^\infty$ in $X$ such that $x_n\in \lambda^{-n} X$ for each $n \in \Nat$ and $x_n \to x$.
\end{proof}

\begin{cor}\label{cor:netsInSSspaces}
For every pointed self-similar doubling homogeneous metric space $\ex$ and every net $X$ in $\ex$ we have $\Lip_0(\ex)\simeq \Lip_0(X)$.
\end{cor}
\begin{proof}
We suppose, without loss of generality, that $\ex$ is isometric to $4\ex$. We will inductively construct a net $X$ in $\ex$ which is $4$-closed. Let $X_1$ be a maximal $1$-separated set in $B(0,2)\setminus B(0,1)$. Set $X'_2=4X_1\subseteq B(0,8)\setminus B(0,4)$. Let $X_2$ be any subset of $B(0,8)\setminus B(0,2)$ containing $X'_2$ so that $X_2\cup X_1$ is a maximal $1$-separated set in $B(0,8)\setminus B(0,1)$. We continue analogously; in particular for $n\geq 2$, $4X_{n-1}\subseteq X_n\subseteq B(0,2\cdot 4^{n-1})\setminus B(0,2\cdot 4^{n-2})$ and $\bigcup_{m\leq n} X_m$ is a maximal $1$-separated set in $B(0,2\cdot 4^{n-1})\setminus B(0,1)$. It is easy to check that $X=\{0\}\cup\bigcup_n X_n$ is $1$-dense, $1$-separated and $4$-closed.
Then we apply Theorem \ref{thm:TheIsomorphismTheorem} to get $\Lip_0(\ex)\simeq \Lip_0(X)$. Finally, if $Y\subseteq\ex$ is an arbitrary net in $\ex$, by \cite[Proposition 5]{HN} (using that $\ex$ is unbounded and separable, hence all the nets are infinite and countable), we get that $\Lip_0(Y)\simeq\Lip_0(X)$, and we are done.
\end{proof}

\begin{remark}Let us note that the only property of doubling metric spaces we used in this section is that there exists $C>0$ such that for each subset $N\subset M$ there exists $E\in\Exto(N,M)$ with $\|E\|\leq C$. Quite an impressive number of results devoted to the non/existence of such $C>0$ for various classes of metric spaces is contained in \cite{BruBru}, where another approach than the one from \cite{LN05} is developed and many interesting results are obtained.
\end{remark}

\section{Lipschitz spaces over finitely generated and Lie groups}

In this section we show that the applicability of our tools to show that $\Lip_0(\zet^d)$ is isomorphic with $\Lip_0(\Rea^d)$, for every $d\in\Nat$, is much wider and besides $\Rea^d$ and its discrete subsets one can consider much larger class of natural doubling metric spaces and its canonical nets coming from analysis. The aim of this section is to consider good pairs $(\Gamma,G)$, where $\Gamma$ is a finitely generated group and $G$ is a (simply connected) Lie group. The restriction to Lie groups as (finite-dimensional) metric spaces is not accidental. Indeed, in order to apply Theorem \ref{thm:TheIsomorphismTheorem} and Lemma \ref{lem:netsInSSspaces} we need to work with pointed doubling self-similar homogeneous metric spaces. If one adds the rather natural assumption of being geodesic, there is a precise characterization of such spaces, proved by Le Donne in \cite{lD}:  these are a special class of Lie groups - the \emph{Carnot groups} equipped with a Carnot-Carath\' eodory distance - we refer the reader to \cite{lD2} for an introduction to Carnot groups. The geometric similarity of Carnot groups with the standard finite-dimensional Euclidean spaces is what makes them an object of vital study aiming to generalize many standard analytic and geometric results valid for Euclidean spaces. The Pansu's celebrated theorem which generalizes the Rademacher theorem to Carnot groups is a prime example (see \cite{Pansu2}).

Before stating our results, let us briefly argue how to view finitely generated groups and Lie groups as interesting examples of pointed metric spaces for which it is natural to study the corresponding spaces of Lipschitz functions or Lipschitz-free spaces. It is the founding idea of geometric group theory to view finitely generated groups as pointed metric spaces. Indeed, if $\Gamma$ is a finitely generated group and $S$ is some finite symmetric (i.e. $s\in S$ if and only if $s^{-1}\in S$) generating set, then we may define a \emph{word metric} $d_S$ associated to $S$ and it is a well-known basic fact that whenever $S'$ is another finite symmetric generating set of $\Gamma$, then the metrics $d_S$ and $d_{S'}$ are bi-Lipschitz equivalent. We refer the interested reader to \cite[Section 7.9]{DrutuKapovich} for further details. Hence, to each finitely generated group $\Gamma$ we may associate a canonical space of Lipschitz functions $\Lip_0(\Gamma)$, a canonical Lipschitz-free space $\F(\Gamma)$, and isomorphism classes of these Banach spaces are well-defined.
Similarly, if $G$ is a \emph{Lie group}, i.e. a group which is a smooth manifold equipped with group operations that are smooth maps, we can always define  a left-invariant Riemannian metric (or rather left-invariant distance function associated to that Riemannian metric). Moreover, all such left-invariant Riemannian metrics, provided that $G$ is a connected real Lie group, are bi-Lispchitz equivalent. We refer to \cite[Section 5.6.3 and Example 2.36 (2)]{DrutuKapovich} for the previous two facts. This again allows us to talk about $\Lip_0(G)$ and $\F(G)$ for such groups, and the isomorphism classes of these Banach spaces are well-defined.
\medskip

Consider a good pair $(\Gamma,G)$ of a finitely generated and a Lie group. There are some restrictions on the pairs. As mentioned above, in order to apply Theorem \ref{thm:TheIsomorphismTheorem}, by the Le Donne's characterization, it is natural to consider $G$ as a Carnot group  and have $\Gamma$ sitting, via a bi-Lipschitz embedding, inside $G$ as a net. With respect to Lemma \ref{lem:bilipofmetrics}, it is natural to consider $\Gamma$ as a subgroup of $G$. This implies that $\Gamma$ must be a finitely generated nilpotent torsion-free group (see \cite[Theorem 2.18]{Raghunathan}). Moreover, vice versa, whenever $\Gamma$ is a finitely generated nilpotent torsion-free group, there exists a unique simply-connected Lie group $G$ such that $\Gamma$ is its subgroup and $G/\Gamma$ is compact (see again \cite[Theorem 2.18]{Raghunathan} and \cite[Theorem 13.40]{DrutuKapovich}). Such $G$ is then called the \emph{Mal'cev closure of $\Gamma$}. The following basic fact is well known and in particular says that $\Gamma$ is a net in its Mal'cev closure.  However, we prove it for the convenience of a reader who is not familiar with lattices in Lie groups.
\begin{fact}\label{fact:cocompactlattices}
Let $G$ be a group with a proper metric $d$ and $\Gamma\leq G$ a discrete subgroup. Then $\Gamma$ is a net in $G$ if and only if $G/\Gamma$ is compact.
\end{fact}
\begin{proof}
Suppose that $\Gamma$ is a net. Then since it is $\varepsilon$-dense in $G$, the quotient $G/\Gamma$ is a continuous image of the compact ball $B_G(1_G,\varepsilon)$. Conversely, if $\Gamma$ is $\varepsilon$-dense for no $\varepsilon>0$, there is a sequence $(g_n)_n\subseteq G$ with $d(g_n,\Gamma)\to\infty$. Their equivalence classes in the metric quotient $G/\Gamma$ form an unbounded set, so $G/\Gamma$ is not compact.
\end{proof}

We shall never explicitly use the group-theoretic properties such as being nilpotent, therefore we will not define it here and refer the reader to any textbook on group theory (our reference in general for geometric group theory is \cite{DrutuKapovich}). However, for a geometrically oriented reader who is not versed in group theory we note that by the celebrated result of Gromov (\cite{Gromov}), $\Gamma$ being finitely generated and torsion-free nilpotent group is essentially the same as having polynomial growth (i.e. the cardinality of balls in $\Gamma$ grows polynomially). 

The previous discussion leads us to consider pairs $(\Gamma,G)$, where $G$ is the Mal'cev closure of $\Gamma$ and $G$ is isomorphic to a Carnot group. Actually, given $\Gamma$, there is a canonical Carnot group associated to it and it is quite often the case that this Carnot group is isomorphic to the Mal'cev closure of $\Gamma$. Let us give some more details. By Pansu \cite{Pansu}, for every finitely generated nilpotent torsion-free group $\Gamma$ with its word metric, the sequence of rescaled balls $B_{n\Gamma}(e,n)$ converges in the Gromov-Hausdorff distance to the unit ball of a certain Carnot group, further denoted by $G_\infty$. Recall that the Gromov-Hausdorff distance between two metric spaces $(M,d)$ and $(N,e)$ is defined as 
\[\begin{split}
d_{GH}(M,N) := \inf\{d_{H,\rho}(M,N)\setsep & \rho \text{ is a metric on } M\cup N \\
& \text{ extending  the metrics on $M$ and $N$}\},
\end{split}\]
where each $d_{H,\rho}(M,N)$ denotes the Hausdorff distance between $M$ and $N$ in $(M\cup N,\rho)$, that is, 
\[
d_{H,\rho}(M,N):=\max\Big\{\sup_{m\in M} \rho(m,N),\sup_{n\in N} \rho(M,n)\Big\}.
\]
The significant property of the Carnot group $G_\infty$ is that it is bi-Lipschitz isomorphic to any asymptotic cone of $\Gamma$. We refer the interested reader to \cite{BrlD13} for more details. In this text we will say that $G_\infty$ is the \emph{Carnot group associated to $\Gamma$}. In many cases, but not always (we will comment on that later), the Mal'cev closure $G$ is isomorphic to the Carnot group $G_\infty$.

The simplest case is $\Gamma=\zet^d$ as then both the Mal'cev closure $G$ and the Carnot group associated to $\zet^d$ is $\Rea^d$. The most simple non-abelian example is probably the discrete Heisenberg group $H_3(\zet)$ which can be modeled as a group of matrices of the form
\[
\begin{pmatrix}
1 & a & c\\
0 & 1 & b\\
0 & 0 & 1
\end{pmatrix}
\]
where $a,b,c\in\zet$. Then the corresponding Mal'cev closure and the Carnot group is the real Heisenberg group $H_3(\Rea)$ which can be modeled exactly the same, however the coefficients $a,b,c$ are real. We remark that the groups $H_3(\zet)$ and $H_3(\Rea)$ are the first known examples of doubling proper metric spaces that do not bi-Lipschitz embed into $L_1$, see \cite{ChKl}. The bi-Lipschitz non-embeddability of $H_3(\zet)$ had been an active research area before (see \cite{CK06} for non-embeddability into all Banach spaces with the Radon-Nikod\' ym property) and it continues to be relevant (see \cite{NL14} and \cite{NY18} for the recent developments). Notice that we immediately get that $\F(H_3(\zet))\not\simeq \F(\zet)$ and $\F(H_3(\Rea))\not\simeq \F(\Rea)$. However, we are unable to say at this moment anything about the relation of $\Lip_0(H_3(\zet))$ and $\Lip_0(H_3(\Rea))$ with $\Lip_0(\Rea^d)$ for any $d\geq 1$.

The following is the main result of this section whose proof is, using the tools developed in the previous section, rather straightforward.
\begin{thm}\label{thm:LipofCarnot}
Let $\Gamma$ be a finitely generated nilpotent torsion-free group, $G$ its Mal'cev closure and $G_\infty$ its associated Carnot group. Suppose that $G$ and $G_\infty$ are isomorphic. Then $\Lip_0(\Gamma)\simeq \Lip_0(G_\infty)$ ($=\Lip_0(G)$).
\end{thm}
Let us prove a simple lemma first. Recall that a proper metric space is such that its closed balls are compact.
\begin{lemma}\label{lem:bilipofmetrics}
Let $G$ be a group with a left-invariant metric $d$ such that $(G,d)$ is proper and geodesic, and let $\Gamma\leq G$ be a subgroup which is a net in $(G,d)$. Then the restriction of $d$ on $\Gamma$ is bi-Lipschitz equivalent to the word metric on $\Gamma$.
\end{lemma}
\begin{proof}
Let $\Gamma$ act on $(G,d)$ by the left multiplication. This is obviously an action by isometries which is moreover geometric. That is, it is cobounded (this follows since $\Gamma$ is $\varepsilon$-dense for some $\varepsilon>0$) and it is properly discontinuous, meaning that for every compact  $B\subseteq G$, the set $\{\gamma\in\Gamma\setsep \gamma\cdot B\cap B\neq \emptyset\}$ is finite. This follows since $\Gamma$ is $\delta$-separated for some $\delta>0$. Now it follows from the Milnor-\v Svarc theorem (see \cite[Theorem 8.37(2)]{DrutuKapovich}) that the word metric on $\Gamma$ and $d$ restricted on $\Gamma$ are quasi-isometric. It is then a basic fact left for the reader that a bijection between two uniformly discrete metric spaces, the identity on $\Gamma$ in our case, is quasi-isometric if and only if it is bi-Lipchitz.
\end{proof}
\begin{proof}[Proof of Theorem \ref{thm:LipofCarnot}]
Fix some $\Gamma$, $G$ and $G_\infty$ as above. Since $\Gamma\leq G$ is a net by Fact \ref{fact:cocompactlattices}, we get that the metric on $\Gamma$ obtained as a restriction of the Carnot-Carath\' edory distance on $G_\infty$ is bi-Lipschitz equivalent with a word metric on $\Gamma$, applying Lemma \ref{lem:bilipofmetrics}.

It then follows by Lemma \ref{lem:netsInSSspaces} that $(\Gamma,G_\infty)$ is a good pair. We are now in position to apply Theorem \ref{thm:TheIsomorphismTheorem} to conclude that $\Lip_0(\Gamma)\simeq\Lip_0(G_\infty)$.
\end{proof}

\begin{remark}
Let us note that it would be sufficient to assume in the statement of Theorem \ref{thm:LipofCarnot} that $\Gamma$ bi-Lipschitz embeds into $G_\infty$ as a net. However, there are indications that it would be no more general. First of all, although we cannot prove it in general, all such `natural' bi-Lipschitz embeddings would imply that $\Gamma$ (with its word metric) would be quasi-isometric to $G_\infty$. This is for example always the case when $\Gamma$ embeds as a subgroup. In particular, $\Gamma$ is quasi-isometric to $G$. If $G$ and $G_\infty$ were not isomorphic, it would therefore give a pair of quasi-isometric but not isomorphic nilpotent Lie groups. That would be against one of the main conjectures in metric geometry of nilpotent Lie groups (see e.g. \cite[Problem 25.44]{DrutuKapovich}).
\end{remark}

A characterization of those finitely-generated nilpotent torsion-free groups $\Gamma$ for which $G$ is isomorphic to  $G_\infty$ was provided by de Cornulier in \cite[Theorem 1.8]{deCor} as those groups for which its systolic growth is asymptotically equivalent to its word growth. We refer the reader therein for an explanation of these notions and for a review when this happens in small dimensions.

Theorem~\ref{thm:LipofCarnot} as a direct generalization of Theorem~\ref{cor:RAndZIso} in particular applies to Euclidean spaces or to the Heisenberg matrices mentioned above.
\begin{cor}
For any $d\in\Nat$ we have $\Lip_0(\zet^d)\simeq \Lip_0(\Rea^d)$. Also, $\Lip_0(H_3(\zet))\simeq \Lip_0(H_3(\Rea))$.
\end{cor}

Quite differently than in infinite-dimensional Banach spaces, where all nets are bi-Lipschitz equivalent, it has been shown that for every $d\geq 2$, $\Rea^d$ contains bi-Lipschitz non-equivalent nets (see \cite{BuKl}); in particular, nets not bi-Lipschitz equivalent with $\zet^d$. Recently, this last result has been generalized to the realm of connected simply connected nilpotent Lie groups (see \cite{DKLL}), where the authors show that each such group with left-invariant Riemannian metric contains bi-Lipschitz non-equivalent nets. Taking into account the result of H\' ajek and Novotn\' y that for every unbounded separable metric space $M$ and nets $\Net_1,\Net_2\subseteq M$, we have $\Lip_0(\Net_1)\simeq\Lip_0(\Net_2)$ (see \cite[Proposition 5]{HN}), we record the following corollary.
\begin{cor}
Let $\Gamma$, $G$ and $G_\infty$ be as in Theorem \ref{thm:LipofCarnot}, and again suppose that $G$ and $G_\infty$ are isomorphic. Then for every net $\Net\subseteq G$ we have $\Lip_0(\Net)\simeq\Lip_0(G)$.
\end{cor}

In contrast to Theorem \ref{thm:LipofCarnot} relating the spaces of Lipschitz functions on a connected Lie group and its discrete lattice, we have the following observation for their preduals.

\begin{fact}\label{prop:freeOverRdAndZdnotIso}
Let $M$ be a metric space such that $[0,1]$ embeds bi-Lipschitz into $M$. Let $N$ be a countable metric space which is uniformly discrete or proper. Then $\F(M)\not\simeq \F(N)$.

In particular, for any $d\in\Nat$, $\F(\zet^d)\not\simeq \F(\Rea^d)$. Or even more generally, for any finitely generated nilpotent torsion-free group $\Gamma$ and its Carnot group $G_\infty$, $\F(\Gamma)\not\simeq \F(G_\infty)$.
\end{fact}
\begin{proof}Since $\F([0,1])$ is isometric to $L_1$, we have $L_1\hookrightarrow \F(M)$. Therefore, $\F(M)$ is a separable Banach space which fails the Radon-Nikod\'ym property and in particular, it is not isomorphic to a dual space.
On the other hand, by \cite[Proposition 4.4]{K04}, if $N$ is a uniformly discrete space, it has the Radon-Nikod\'ym property and by \cite[Theorem 2.1]{Da}, if $N$ is countable and proper, it is a dual space.

The ``in particular'' part follows, which is clear since every  finitely generated nilpotent torsion-free group $\Gamma$ is a countable proper metric space and since $G_\infty$ is geodesic, $[0,1]$ embeds isometrically into $G_\infty$.
\end{proof}

\begin{remark}
Notice that this generalizes the old result that $\ell_1$ and $L_1([0,1])$ are not isomorphic, while the spaces $\ell_\infty$ and $L_\infty([0,1])$ are (\cite{Pe}). It follows from results of Christensen that no isomorphism between $\ell_\infty$ and $L_\infty([0,1])$ is $w^*$-measurable and so it cannot be explicitly definable, see e.g. \cite[Chapters 5 and 6]{Ch74}. It could be interesting to investigate whether the same phenomenon occurs for isomorphisms between $\Lip_0(\Gamma)$ and $\Lip_0(G)$, where $\Gamma$ and $G$ are from the statement of Theorem \ref{thm:LipofCarnot}.
\end{remark}

One may wonder what we can prove with our methods when for a given $\Gamma$ the corresponding Lie groups $G$ and $G_\infty$ are not isomorphic. We can still deduce the following.
\begin{prop}\label{michal}
Let $\Gamma$, $G$ and $G_\infty$ be as in Theorem \ref{thm:LipofCarnot}. Then $\Lip_0(G_\infty)\complemented\Lip_0(\Gamma)\complemented\Lip_0(G)$.
\end{prop}

\begin{proof}
First of all, since $\Gamma\leq G$ and $G/\Gamma$ is compact, we get, using Lemma \ref{lem:bilipofmetrics}, that $\Gamma$ with its word metric bi-Lipschitz embeds into $G$. So by Proposition \ref{prop:complInRd}, since $\Gamma$ is a doubling subspace of $G$, we get $\Lip_0(\Gamma)\complemented\Lip_0(G)$. So we must show that $\Lip_0(G_\infty)\complemented\Lip_0(\Gamma)$.

By the above mentioned Pansu's result \cite{Pansu}, for each $r>0$ the sequence $(B_{n\Gamma}(e,rn))_{n=1}^\infty$ converges in the Gromov-Hausdorff distance to the ball $B_{G_\infty}(e,r)$.  Moreover, since $\Gamma$ is a homogeneous and doubling subspace of $G$, by Theorem~\ref{thm:isoToEllInftySum} and Proposition~\ref{prop:complInRd} we get $\Lip_0(\Gamma)\simeq \bigoplus_{\ell_\infty}\Lip_0(\Gamma)$. Hence, it suffices to prove the following Lemma.
\begin{lemma}
Let $(M,d_M)$ and $(N,d_N)$ be pointed metric spaces such that $N$ is doubling, $M$ is unbounded and for each $r>0$ the sequence $(B_{nM}(0,rn))_{n=1}^\infty$ converges in the Gromov-Hausdorff distance to the ball $B_{N}(0,r)$. Then $\Lip_0(N)\complemented \bigoplus_{\ell_\infty}\Lip_0(M)$.
\end{lemma}
\begin{proof}For every $m\in\Nat$ we can find $n_m\in\Nat$ so that the balls $B_{N}(0,m)$ and $B_{n_mM}(0,mn_m)$ are of Gromov-Hausdorff distance less than $1/m$. So, by definition, we can find a metric $d_m$ on $B_{N}(0,m)\cup B_{n_mM}(0,mn_m)$ that extends the metrics on $B_{N}(0,m)$, resp. $B_{n_mM}(0,mn_m)$, and such that for every $y\in B_{N}(0,m)$ there is $x\in B_{n_mM}(0,mn_m)$ such that $d_m(y,x)<1/m$, and vice versa. For every $m\in\Nat$, we select some $4/m$-net $\Net_m$ containing $0$ in $B_{N}(0,m)$, i.e. a maximal $4/m$-separated subset of $B_{N}(0,m)$, and find some one-to-one map $j_m:\Net_m\rightarrow B_{n_mM}(0,mn_m)$ such that $j_m(0)=0$ and for every $y\in\Net_m$ we have $d_m(y,j_m(y))<1/m$. Let $\Net'_m$ be the image of $j_m$. Note that $j_m$ is a bi-Lipschitz bijection with $\Net'_m$. Indeed, for any $y,y'\in\Net_m$, $y\neq y'$ we have 
\begin{align*}
\frac{d_{n_mM}(j_m(y),j_m(y'))}{d_{N}(y,y')}=\frac{d_m(j_m(y),j_m(y'))}{d_m(y,y')}
\leq \frac{d_m(y,y')+2/m}{d_m(y,y')}\leq \frac{6/m}{4/m}=\frac{3}{2}.
\end{align*}
On the other hand we have
\[\begin{split}
\frac{d_{n_mM}(j_m(y),j_m(y'))}{d_{N}(y,y')} & =\frac{d_m(j_m(y),j_m(y'))}{d_m(y,y')}\geq \frac{d_m(y,y')-2/m}{d_m(y,y')}\\
& \geq \frac{d_m(y,y')-d_m(y,y')/2}{d_m(y,y')}=\frac{1}{2}.
\end{split}\]
It follows that $\Net'_m$ is a net in $B_{n_mM}(0,mn_m)$. Let $i_m:\Net'_m\rightarrow \Net_m$ be the inverse of $j_m$. Clearly, for every $y\in N$ there exists a sequence $(y_m)_{m=1}^\infty$, where $y_m\in\Net_m$ and $d_{N}(y_m,y)\to 0$. Therefore, $(j_m(y_m))_{m=1}^\infty$ is a sequence with $i_m(j_m(y_m))\to y$.
So we are in position to apply Key Lemma \ref{lem:key} with $(i_m)_m$ as above to get that $\Lip_0(N)\complemented\bigoplus_{\ell_\infty} \Lip_0(\Net'_m)$. Now each $\Net'_m$ is a closed subset of $n_mM$ and it is doubling with uniformly bounded doubling constant, because it is a bi-Lipschitz image (with uniformly bounded bi-Lipschitz constant) of the set $\Net_m$ which is a doubling space as a subset of the doubling space $N$. Hence, by Proposition~\ref{prop:complInRd}, there is a uniform constant $C$ such that each $\Lip_0(\Net_m)$ is $C$-complemented in $\Lip_0(n_mM)$; hence, we have $\Lip_0(N)\complemented\bigoplus_{\ell_\infty} \Lip_0(n_mM)$ and it suffices to observe that each $\Lip_0(n_mM)$ is isometric to $\Lip_0(M)$ via the mapping $\Lip_0(n_mM)\ni f\mapsto n_mf\in\Lip_0(M)$.
\end{proof}
\end{proof}

\section{Lipschitz spaces over infinite-dimensional Banach spaces}\label{sec:banach}

Even though our tools were originally developed in order to prove that $\Lip_0(\zet^d)\simeq\Lip_0(\Rea^d)$ for each $d\in\Nat$, they may be used in order to get some interesting consequences concerning $\Lip_0(X)$ spaces, where $X$ is an infinite-dimensional Banach space.

\subsection{Statements of a more general nature}

\begin{thm}\label{thm:lipComplemented}Let $X$ and $Y$ be separable Banach spaces such that there is $C>0$ and an increasing sequence $(X_n)_{n=1}^\infty$ of subspaces of $X$ with $\overline{\bigcup_{n=1}^\infty X_n} = X$ such that each $X_n$ is $C$-isomorphic to a $C$-complemented subspace of $Y$. Then 
\[
\Lip_0(X)\complemented \Lip_0(Y).
\]
Moreover, if $Y=X$ (that is, each $X_n$ is $C$-complemented in $X$), then
\[
\Lip_0(X)\simeq \bigoplus_{\ell_\infty} \Lip_0(X_n).
\]

In particular, if $X$ is a separable infinite-dimensional Banach space then
\begin{itemize}
\item if $(e_n)_{n=1}^\infty$ is a Schauder basis of $X$ and $X_n = \Span\{e_i\setsep i\leq n\}$, then $\Lip_0(X)\simeq \bigoplus_{\ell_\infty} \Lip_0(X_n)$;
\item if $X$ has no nontrivial cotype, then $\Lip_0(c_0)\complemented \Lip_0(X)$;
\item if $X$ has nontrivial type, then $\Lip_0(\ell_2)\complemented \Lip_0(X)$.
\end{itemize}
\end{thm}
\begin{proof}By Lemma~\ref{lem:key}, we have $\Lip_0(X)\complemented\bigoplus_{\ell_\infty}\Lip_0(X_n)$. By the assumptions we have $\bigoplus_{\ell_\infty}\Lip_0(X_n)\complemented \bigoplus_{\ell_\infty} \Lip_0(Y)$. By Theorem~\ref{thm:kaufman}, we have $\bigoplus_{\ell_\infty} \Lip_0(Y)\simeq \Lip_0(Y)$ and putting all together we obtain $\Lip_0(X)\complemented \Lip_0(Y)$ and $\Lip_0(X)\simeq \bigoplus_{\ell_\infty} \Lip_0(X_n)$ if $Y=X$.

The ``In particular'' part follows from certain known results. The result concerning Schauder basis is immediate. Further, by the result of B. Maurey and G. Pisier \cite{MaPi76}, $X$ has no nontrivial cotype iff $\ell_\infty$ is finitely representable in $X$, and since each $\ell_\infty^n$ is $1$-injective, we easily see that for each $n$, $\ell_\infty^n$ is $2$-isomorphic to a $2$-complemented subspace of $X$, which by the above implies $\Lip_0(c_0)\complemented \Lip_0(X)$. If $X$ has nontrivial type, we use the result by T. Figiel and N. Tomczak-Jaegermann \cite{FiTo79} (see also \cite[Theorem 15.10]{MiSch86}), by which there is $C>0$ such that, for each $n$, $\ell_2^n$ is $C$-isomorphic to a $C$-complemented subspace of $X$; hence, $\Lip_0(\ell_2)\complemented \Lip_0(X)$.
\end{proof}

\begin{remark}\label{rem:freeIzo}By Theorem~\ref{thm:lipComplemented}, we have $\Lip_0(c_0)\simeq \bigoplus_{\ell_\infty}\Lip_0(\ell_\infty^n)$. However, it is not true that $\F(c_0)\simeq \bigoplus_{\ell_1} \F(\ell_\infty^n)$. Indeed, by the result of T. Kochanek and E. Perneck\'a \cite{KP}, the Banach spaces $\F(\ell_\infty^n)$ enjoy the Pe\l czy\'nski's property $(V^*)$. This property is preserved by $\ell_1$ sums \cite{Bo89} and it implies that the Banach space is weakly sequentially complete and so it does not contain isomorphic copy of $c_0$; hence, we have $c_0\not\hookrightarrow \bigoplus_{\ell_1} \F(\ell_\infty^n)$. But, by the result of G. Godefroy and N. Kalton \cite{GK}, we have $c_0\hookrightarrow \F(c_0)$.
\end{remark}

Let us recall that a Banach space $X$ is said to be an $\mathcal L_{p,\lambda}$-space (with $1 \leq  p \leq \infty$ and $\lambda\geq 1$) if every finite-dimensional subspace of $X$ is contained in an $n$-dimensional subspace of $X$, for some $n\in\Nat$, whose Banach-Mazur distance to $\ell_p^n$ is at most $\lambda$. A space $X$ is said to be an $\mathcal L_p$-space if it is an $\mathcal L_{p,\lambda}$-space for some $\lambda\geq 1$.
\begin{thm}\label{thm:scriptLp}Let $X$ be a separable infinite dimensional $\mathcal{L}_p$-space for some $p\in [1,\infty]$. Then $\Lip_0(X)\simeq\Lip_0(\ell_p)$ if $p< + \infty$ and $\Lip_0(X)\simeq\Lip_0(c_0)$ if $p=+\infty$.
\end{thm}
\begin{proof}It follows from the definition of an $\mathcal L_p$-space and from Theorem~\ref{thm:lipComplemented} that we have $\Lip_0(X)\complemented \Lip_0(\ell_p)$. On the other hand, by \cite[Corollary 2.1]{PeRo75}, there exists $C>0$ and an increasing sequence $(k_n)_{n=1}^\infty$ such that each $\ell_p^{k_n}$ is $C$-isomorphic to a $C$-complemented subspace of $X$; hence, by Theorem~\ref{thm:lipComplemented}, we obtain $\Lip_0(\ell_p)\complemented \Lip_0(X)$.

The conclusion follows from the Pe\l czy\'nski decomposition method and the fact that $\Lip_0(\ell_p)$ is isomorphic to its $\ell_\infty$--sum.
\end{proof}

\begin{remark}It is known \cite{DF} that whenever $K$ is infinite metric and compact space then we have $\F(\mathcal{C}(K))\simeq \F(c_0)$; in particular, $\Lip_0(\mathcal{C}(K))\simeq \Lip_0(c_0)$. Theorem~\ref{thm:scriptLp} might be seen as a far reaching generalization of the later fact which naturally leads to the question of whether $\F(X)$ and $\F(c_0)$ are isomorphic whenever $X$ is a separable infinite dimensional $\mathcal{L}_\infty$-space.
\end{remark}

Our next application is based on a result by P. H\'ajek and M. Novotn\'y \cite{HN}. Recall that all nets in a given infinite dimensional Banach space are bi-Lipschitz equivalent, see \cite{LMP}.

\begin{cor}\label{thm:hajekNovotny}Let $X$ be a Banach space such that $X\simeq Y\oplus X$, where $Y$ is a Banach space with a Schauder basis. Let $\mathcal{N}_X$ be a net in $X$. Then $\Lip_0(X)\complemented\Lip_0(\mathcal{N}_X)$.
\end{cor}
\begin{proof}By Lemma~\ref{lem:netsInSSspaces} and Proposition~\ref{prop:PropForGoodPairs}, we have $\Lip_0(X)\complemented \bigoplus_{\ell\infty} \Lip_0(\mathcal{N}_X)$. Since $X\simeq Y\oplus X$, we may apply \cite[Theorem 8]{HN} and get $\bigoplus_{\ell_1} \F(\mathcal{N}_X)\simeq \F(\mathcal{N}_X)$ which implies $\bigoplus_{\ell_\infty} \Lip_0(\mathcal{N}_X)\simeq \Lip_0(\mathcal{N}_X)$.
\end{proof}
Note that, as mentioned already in \cite{HN}, the above applies to quite a rich class of spaces.

\begin{remark}\label{rem:grot}Let us note two consequences of the above which we consider interesting.
\begin{itemize}
\item $\Lip_0(\mathcal{N}_{c_0})$ is not a Grothendieck space. Indeed, by Theorem~\ref{thm:hajekNovotny} we have $\Lip_0(c_0)\complemented \Lip_0(\mathcal{N}_{c_0})$ and $\Lip_0(c_0)$ is not Grothendieck as the predual $\F(c_0)$ contains $c_0$ and so it is not weakly sequentially complete.
\item $\Lip_0(\mathcal{N}_{\ell_1})$ is not a Grothendieck space. Indeed, by Theorem~\ref{thm:hajekNovotny} we have $\Lip_0(\ell_1)\complemented \Lip_0(\mathcal{N}_{\ell_1})$ and $\Lip_0(\ell_1)$ is not Grothendieck because by \cite{Da13} we have $\ell_1\complemented \Lip_0(\ell_1)$.
\end{itemize}
\end{remark}

In the light of Corollary~\ref{thm:hajekNovotny} it is natural to ask for which separable infinite-dimensional Banach spaces $X$ we have $\Lip_0(X)\simeq\Lip_0(\mathcal{N}_X)$. The only result in this direction we have is the following, which we prove in the remainder of this subsection.

In the sequel, for every $n\in\Nat$, $\zet^n$ will be viewed as a canonical subset of $\ell_1$, i.e. the subgroup generated by the first $n$ basis vectors, and as such it will be equipped with the inherited $\ell_1$-metric. Analogously $\bigoplus_\Nat \zet$, the subgroup generated by all the basis vectors of $\ell_1$, which will be for notational reasons denoted by $\zet^{<\Nat}$.
\begin{prop}\label{prop:netyBanach}We have 
\[
\bigoplus_{\ell_\infty}\Lip_0(\zet^{<\Nat})\simeq \Lip_0(\ell_1).
\]
\end{prop}

In the proof we will use an important geometrical observation by E. Perneck\'a and G. Lancien \cite{LP13}. Let us denote by $C_n(y,R)$, where $n\in\Nat$, $y\in\Rea^n$ and $R\in (0,\infty)$, the hypercube having edge length $R$ and vertices $v_\gamma$, $\gamma\in\{0,1\}^n$, given by $v_\gamma = y + R\gamma$, that is 
\[C_n(y,R):=\operatorname{co}\{v_\gamma\setsep \gamma\in\{0,1\}^n\}.\]
If $C$ is a hypercube as above, we denote by $V_C$ its vertices. Given a hypercube $C\subset \Rea^n$ as above and a function $f$ defined on $V_C$, there exists a unique function $\Lambda(f,C)$ defined on $C$ which is extending $f$ and is coordinatewise affine, i.e., $t \mapsto \Lambda(f,C)(x_1,\ldots, x_{i-1}, t, x_{i+1},\ldots, x_n)$ is affine whenever $1 \leq i \leq n$. The function $\Lambda(f,C)$ was originally constructed inductively, see \cite[Section 3.1]{LP13}, but there exists even an explicit formula mentioned in \cite[Section 3]{PS}
\begin{equation}\label{eq:interpolationFunction}
	\Lambda(f,C)(x):=\sum_{\gamma\in\{0,1\}^n}\left(\prod_{i=1}^n \left(1 - \gamma_i + (-1)^{\gamma_i+1}\frac{x_i-y_i}{R}\right)\right)f(v_\gamma).
\end{equation}
The crucial result for us will be the following result proved in \cite[Lemma 3.2]{LP13}, for an alternative proof see also \cite[Remark 3.4]{PS}.

\begin{lemma}\label{lem:perneckaLancien}Let $C\subset \Rea^n$ be a hypercube. Consider $\Rea^n$ equipped with the $\ell_1$-norm. Then for every function $f:V_C\to \Rea$ we have \[\Lip(\Lambda(f,C)) = \Lip(f).\]
\end{lemma}

 \begin{cor}\label{cor:ellOneExtension}Let $n\in\Nat$ and consider $\zet^n$ as a metric subspace of $\ell_1^n$. Then there exists $E\in\Ext(\zet^n,\ell_1^n)$ with $\|E\|\leq 1$.
 \end{cor}
 \begin{proof}We cover $\ell_1^n$ by hypercubes $C_n(x,1)$, $x\in\zet^n$. Further, we put $E(f)(y)=\Lambda(f,C_n(x,1))(y)$ whenever $y\in C_n(x,1)$ for some $x\in\zet^n$. Since $\Lambda(f,C)$ is uniquely determined as the only coordinatewise affine extension of $f|_{V_C}$, we easily see that $E$ is well-defined and, having in mind the formula \eqref{eq:interpolationFunction}, we see that $E$ is linear and pointwise-to-pointwise continuous. Moreover, by partitioning each segment $[x, y]$ with respect to the hypercubes through which it passes and using the estimate from Lemma~\ref{lem:perneckaLancien}, we get $\|E(f)\|_{\Lip}\leq \|f\|_{\Lip}$ for each $f\in \Lip_0(\zet^n)$.
 \end{proof}

\begin{proof}[Proof of Proposition~\ref{prop:netyBanach}]
Since $(\zet^{<\Nat},\ell_1)$ is obviously a good pair, by Proposition \ref{prop:PropForGoodPairs} we have $\Lip_0(\ell_1)\complemented \bigoplus_{\ell_\infty}\Lip_0(\zet^{<\Nat})$. Using our Key Lemma~\ref{lem:key}, $\Lip_0(\zet^{<\Nat})$ is isometric to a 1-complemented subspace of $\bigoplus_{\ell_\infty}\Lip_0(\zet^n)$. By Corollary~\ref{cor:ellOneExtension} and Fact~\ref{fact:extensions}, each $\Lip_0(\zet^n)$ is isometric to a 1-complemented subspace of $\Lip_0(\ell_1^n)$(note that, by Theorem~\ref{cor:RAndZIso}, we also have $\Lip_0(\zet^n)\simeq \Lip_0(\ell_1^n)$, but now the important information is that the norms of the isomorphisms and projections are uniformly bounded) which is isometric to a 1-complemented subspace of $\Lip_0(\ell_1)$ by Fact~\ref{fact:retractionImpliesComplementedLipSpaces}. Putting all together, we get
\[
	\Lip_0(\ell_1)\complemented \bigoplus_{\ell_\infty}\Lip_0(\zet^{<\Nat})\complemented \bigoplus_{\ell_\infty}\Big(\bigoplus_{\ell_\infty}\Lip_0(\zet^n)\Big)\complemented \bigoplus_{\ell_\infty}\bigoplus_{\ell_\infty}\Lip_0(\ell_1)\simeq \Lip_0(\ell_1)
\]
and an application of the Pe\l czy\' nski's decomposition method finishes the proof.
\end{proof}

Let us note that we do not know whether $\Lip_0(\zet^{<\Nat})$ is isomorphic to its $\ell_\infty$-sum.

\subsection{Lipschitz spaces over Orlicz spaces}
The methods developed for good pairs of metric spaces from the first and second sections can be applied even in the infinite-dimensional setting. As we will see, $(\ell_p,L_p)$ is a prime example of a good pair in this context. It turns out however that it is possible to work with sequence and function spaces of greater generality than the standard $(\ell_p,L_p)$-case. Indeed, we are naturally led to the Orlicz spaces, which generalize the $L_p$ case, and the content of this subsection is to exploit our tools in this generality. We start with basic preliminaries.
\medskip

Let $(\Omega,\Sigma,\mu)$ be a $\sigma$-finite measure space and $\Phi$ be a \emph{Young function} satisfying the $\Delta_2$-condition, namely, a function   $\Phi: \left[0,\infty\right)\to \left[0,\infty\right)$ satisfying 

\begin{enumerate}[(1)]
\item $\Phi(0)=0$;
\item\label{cond:two} $\Phi$ is convex and increasing;
\item $\Phi$ satisfies the \emph{$\Delta_2$-condition}: there is $K>0 $ such that $\Phi(2x)\leq K\Phi(x)$ for all $x > 0$.
\end{enumerate}

\begin{remark}
From convexity and condition (1) it follows that a Young function $\Phi$ is continuous at $0$, therefore, it is continuous in $[0,\infty)$. From the condition (2) one may easily deduce that $\displaystyle\liminf_{x\to \infty} \Phi(x)/x > 0$; in particular, $\Phi$ is unbounded. We will call a Young function $\Phi$ an \emph{$N$-function} if it also satisfies 
$\lim_{x\to 0}\frac{\Phi(x)}{x}=0$ and $\lim_{x\to \infty}\frac{\Phi(x)}{x}=\infty$, see \cite[\S II]{Orliczbook}.
\end{remark}

For the sequel, a Young function $\Phi$ is fixed.

Let us denote by $L_\Phi(\Omega,\mu)$ the set of all the real-valued measurable functions $f$ defined $\mu$-almost everywhere on $\Omega$ (functions which agree $\mu$-almost everywhere are identified) such that \[\int_\Omega\Phi(|f|)d\mu<\infty\]
and let 
\[\|f\|_\Phi:=\inf\{r>0\setsep \int_\Omega\Phi(|f|/r)d\mu\leq 1\},\quad f\in L_\Phi(\Omega,\mu)\] 
It is a classical result, see e.g. \cite[Chapter III]{Orliczbook}, that $(L_\Phi(\Omega,\mu),\|\cdot\|_\Phi)$ is a Banach space. We refer the reader to \cite{Orliczbook} for more information. 

If $\Omega$ is some Lebesgue-measurable subset of $\Rea$ and $\mu$ is the Lebesgue measure restricted to $\Omega$, we will denote $L_\Phi(\Omega,\mu)$ by $L_\Phi(\Omega)$. If $\Omega=\Nat$ and $\mu$ is the counting measure on $\Nat$, we will denote $L_\Phi(\Omega,\mu)$ by $\ell_\Phi$, the space of all sequences $(x_n)_{n=1}^\infty\in \Rea^\Nat$ such that 
\[\|(x_n)_n\|_\Phi:=\inf\{r>0\setsep \sum_{i=1}^\infty \Phi(\frac{|x_i|}{r})\leq 1\}<\infty,\] also called \emph{Orlicz sequence space}, defined in \cite{LiTz77}.

\begin{fact}\label{fact: LphisubsetL1} If $\mu(\Omega)<\infty$, then $L_\Phi(\Omega,\mu)\subset L_1(\Omega,\mu)$.
\end{fact}
\begin{proof}
Let $f$ be an arbitrary function from $L_\Phi(\Omega,\mu)$. Since $\Phi$ is a convex increasing function, we have $\liminf_{x\to \infty} \Phi(x)/x > 0$. It means that there exists $\delta>0$ and $N>0$ so that $\Phi(x)/x\geq\delta$ whenever $x\geq N$. Defining $\Omega_N=\{x\in \Omega: |f(x)|\geq N\}$, we may write 
\begin{align*}\int_\Omega|f|d\mu&=\int_{\Omega_N}|f|d\mu+\int_{\Omega\setminus \Omega_N}|f|d\mu \\
                                &\leq \delta^{-1}\int_{\Omega}\Phi(|f|)d\mu+N\mu(\Omega)<\infty, 
\end{align*}
and deduce that $f \in L_1(\Omega,\mu)$.
\end{proof}

The proof of the next well known fact is easy and therefore will be omitted.

\begin{fact}\label{Fact:Orlicz}
If $0<\alpha<\beta$, then the identity map $Id:L_{\alpha \Phi}(\Omega)\to L_{\beta \Phi}(\Omega)$ is an isomorphism with distortion 
$\|Id\|\|Id^{-1}\|\leq \frac{\beta}{\alpha}$.
\end{fact}

\begin{remark}By applying \cite[Theorem 17.41]{Kechris} and by using some standard measure theory arguments, it is easy to check that if $(\Omega,\Sigma)$ is a standard Borel space and $\mu$ is non-atomic $\sigma$-finite measure, then $L_\Phi(\Omega,\mu)$ is isomorphic to either $L_\Phi[0,1]$ or  $L_\Phi[0,\infty)$ depending if $\mu$ is finite or infinite respectively, therefore, in what follows we will restrict our investigations to $L_\Phi[0,1]$ and $L_\Phi[0,\infty)$.
\end{remark}

The main result of this section is as follows. We note that the `in particular' part can be derived already from Theorem \ref{thm:scriptLp}. Further, let us emphasize that by $\ell_{\alpha_n\Phi}$ below we denote the Orlicz space associated to the function $\alpha_n\Phi$.

\begin{thm}\label{thm:OrliczLip}If $(\alpha_n)_n$ is a sequence of strictly positive real numbers with $\displaystyle\lim_{n\to \infty}\alpha_n=0$, then 
\[\Lip_0(L_\Phi[0,\infty))\simeq \bigoplus_{\ell_{\infty}}\Lip_0(\ell_{\alpha_n\Phi}).\] 
In particular, $\Lip_0(\ell_p)\simeq \Lip_0(L_p)$ for $1\leq p <\infty$. 
\end{thm}

From the previous theorem and an application of a classical result we obtain another consequence. 

\begin{thm}\label{thm:OrliczL[0,1]}
If $\Phi$ is a $N$-function and $L_\Phi[0,1]$ is reflexive, then there exists a sequence $(X_n)_n$ of Banach spaces such that 
$X_n\simeq \ell_2$ for each $n \in \mathbb{N}$ and
\[\Lip_0(L_\Phi[0,1])\simeq \bigoplus_{\ell_\infty}\Lip_0(X_n).\]
In particular, each $\Lip_0(L_p)$ for $1< p <\infty$ admits a decomposition as above.
\end{thm}

In order to give a proof of Theorem \ref{thm:OrliczLip} we will first establish a couple of lemmas.

\begin{lemma}\label{lem:Orliczgoodpair} Let $(\alpha_n)_n$ be a sequence of strictly positive real numbers. If $\displaystyle\lim_{n\to \infty}\alpha_n=0$, then \[\Lip_0(L_\Phi[0,\infty))\complemented\bigoplus_{\ell_{\infty}}\Lip_0(\ell_{\alpha_n\Phi}).\]
\end{lemma}
\begin{proof}
For each $m \in \Nat$, consider the operator $T_m:\ell_{\alpha_m\Phi}\to L_\Phi[0,\infty)$ given by 
\[T_m((x_n)_n)=\sum_{n=1}^{\infty}x_n\cdot \chi_{[\alpha_m(n-1),\alpha_m n)}.\]
If $T_m((x_n)_n)=f$, then for each $r>0$ we have
\begin{align*}\int_0^\infty\Phi(\frac{|f|}{r}) d\mu&=\sum_{n=1}^\infty \int_{\alpha_m(n-1)}^{\alpha_m n}\Phi(\frac{|f|}{r}) d\mu\\
&=\sum_{n=1}^\infty \int_{\alpha_m(n-1)}^{\alpha_m n}\Phi(\frac{|x_n|}{r}) d\mu= \sum_{n=1}^\infty\alpha_m\Phi(\frac{|x_n|}{r}).
\end{align*}
It follows that $\|T_m((x_n)_n)\|_\Phi=\|(x_n)_n\|_{\alpha_m\Phi}$ and we deduce that $T_m$ is an isometric embedding of $\ell_{\alpha_m\Phi}$ into $L_{\Phi}[0,\infty)$.

Next, we will prove that for any given $f \in L_{\Phi}[0,\infty)$ and $\varepsilon>0$ there is $m_0\in \Nat$ such that for each $m\geq m_0$ there is $h\in T_m[\ell_{\alpha_m\Phi}]$ with \[\int_0^\infty\Phi(\frac{|f-h|}{\varepsilon})d\mu\leq 1\] 
and it will follow that $\|f-h\|_\Phi\leq \varepsilon$.

Because $f\in L_\Phi[0,\infty)$ and $\Phi$ satisfies the $\Delta_2$-condition on the whole interval $\left[0,\infty\right)$, we deduce that $f/\varepsilon\in L_\Phi[0,\infty)$. Then  $\int_0^\infty\Phi(\frac{|f|}{\varepsilon})d\mu<\infty$ and we may fix $N\in \mathbb{N}$ such that $\int_N^\infty\Phi(\frac{|f|}{\varepsilon})d\mu<\frac{1}{2}$.

From Fact \ref{fact: LphisubsetL1} we have that $f|_{[0,N+1]}$ belongs to $L_1[0,N+1]$. Because $\alpha_n>0$ for each $n$ and $\lim_{n \to \infty}\alpha_n=0$, if $\floor{x}=\max\{m\in \Nat:m\leq x\}$, then $\lim_{n \to \infty}\alpha_n\floor{\alpha_n^{-1}(N+1)}=N+1$ and it is standard to check, by using the Lebesgue Differentiation Theorem, that the sequence of step functions $(h_n)_n$ given by the expression 
\[h_n=\sum_{j=1}^{\floor{\alpha_n^{-1}(N+1)}}a_j\cdot\chi_{[\alpha_n(j-1),\alpha_n j)}+\chi_{[\alpha_n\floor{\alpha_n^{-1}(N+1)},N+1]}\]
where $\displaystyle a_j=\alpha_n\int_{\alpha_n(j-1)}^{\alpha_n j} f d \mu $ for each $j$, converges pointwise to $f$ $\mu$-almost everywhere in $[0,N+1]$.

Since $\Phi$ is continuous at $0$, the sequence $(\Phi\circ \frac{|f - h_n|}{\varepsilon})_{n=1}^\infty$, seen as elements of $L_1[0,N+1]$, converge pointwise to the zero function $\mu$-almost everywhere. Then, by the Lebesgue's Dominated Convergence Theorem, there is $m_0$ such that if $m\geq m_0$, then $\int_0^{N+1}\Phi(\frac{|f-h_m|}{\varepsilon})\leq \frac{1}{2}.$ We may also assume $m_0$ sufficiently large
so that $N\leq \alpha_m\floor{\alpha_m^{-1}(N+1)}$ whenever $m\geq m_0$. Then, if $m\geq m_0$ we consider
\[h=\sum_{j=1}^{\infty}x_j\cdot \chi_{[\alpha_m(j-1),\alpha_m j)}\] 
where $x_j=h_m(\alpha_m(j-1))$ if $j\leq \floor{\alpha_m^{-1}(N+1)}$ and $x_j=0$ for otherwise. Clearly, $h$ is an element 
of $T_m[\ell_{\alpha_m\Phi}]$. Finally,
\begin{align*}\int_0^\infty\Phi(\frac{|f-h|}{\varepsilon})&= \int_0^N\Phi(\frac{|f-h|}{\varepsilon})+\int_N^\infty\Phi(\frac{|f-h|}{\varepsilon})\\
&\leq\int_0^{\floor{\alpha_m^{-1}(N+1)}}\Phi(\frac{|f-h_m|}{\varepsilon})+\int_N^\infty\Phi(\frac{|f|}{\varepsilon})\leq \frac{1}{2}+\frac{1}{2}=1.
\end{align*}
The thesis follow by a direct application of the Key Lemma \ref{lem:key}.
\end{proof}

\begin{remark}
From Fact \ref{Fact:Orlicz}, the spaces $\ell_{\alpha_n \Phi}$ in the previous Lemma are pairwise isomorphic. In the particular case $\Phi(t)=t^p$, for some $1\leq p <\infty$, they are all isometric to $\ell_p$ and the mappings $T_m$ above constitute a sequence of linear embeddings of $\ell_p$ into $L_p[0,\infty)$ with distortion $\|T_m\|\|T_m^{-1}\|$ uniformly bounded by $1$. This fact together with the subsequent argument leads to the conclusion that $(\ell_p,L_p)$ is a good pair.
\end{remark}

\begin{lemma}\label{lem:phicomplemented} For any $\alpha>0$,
$\ell_{\alpha\Phi}$ is isometric to a $1$-complemented subspace of $L_\Phi[0,\infty)$.
\end{lemma}
\begin{proof}
It is well-known and straightforward to check that the expression
\[P(f)=\sum_{n=1}^{\infty}(\frac{1}{\alpha}\int_{\alpha(n-1)}^{\alpha n}fd\mu)\cdot \chi_{[\alpha(n-1),\alpha n)}.\]
defines a $1$-projection from $ L_\Phi$ onto an isometric copy of $\ell_{\alpha \Phi}$.
\end{proof}

We are now in position of proving the main results of this section:

\begin{proof}[Proof of the Theorem \ref{thm:OrliczLip}]
It follows from Lemma \ref{lem:phicomplemented} and Fact~\ref{fact:retractionImpliesComplementedLipSpaces} that $\Lip_0(\ell_{\alpha_n\Phi})$ is $1$-complemented in $\Lip_0(L_\Phi[0,\infty))$ for each $n \in \Nat$. Then, by applying Theorem~\ref{thm:kaufman}, we get $\displaystyle\bigoplus_{\ell_{\infty}}\Lip_0(\ell_{\alpha_n})\complemented \Lip_0(L_\Phi[0,\infty))$. By Lemma~\ref{lem:Orliczgoodpair} and an application of the Pe\l czy\' nski decomposition method we obtain $\displaystyle\bigoplus_{\ell_{\infty}}\Lip_0(\ell_{\alpha_n})\simeq \Lip_0(L_\Phi[0,\infty))$. The ``In particular'' part follows from the easily verifiable fact that for $\Phi(t) = t^p$ we have $\ell_{\alpha_n\Phi}$ isometric to $\ell_p$ for each $n\in\Nat$.
\end{proof}

\begin{remark}\label{Rem:aux}For the next proof, we recall that if $\Phi$ is an $N$-function, according to  \cite[Corollary IV.4.12]{Orliczbook}, $L_\Phi[0,1]$ is reflexive if and only if $\Phi$ satisfies both $\Delta_2$ and $\nabla_2$-condition (i.e., there is some 
$K>1$ so that $\Phi(x)\leq \frac{\Phi(Kx)}{2K}$ for all $x$), which is equivalent, see \cite[Corollary II.2.3]{Orliczbook}, to the fact that both $\Phi$ and its complementary Young function $\Psi$ ($\Psi(t)=\sup_{s>0}\{st-\Phi(s)\}$, $t\geq 0$) satisfy the $\Delta_2$-condition. These conditions imply nontrivial Boyd indexes for $L_\Phi[0,1]$ (see e.g. \cite[Proposition 2.b.2]{LiTzII}), for which \cite[Theorem 8.6]{JoMaSchTz} applies, see also \cite[Theorem 2.f.1]{LiTzII}.
\end{remark}

\begin{proof}[Proof of the Theorem \ref{thm:OrliczL[0,1]}]
Without loss of generality we may assume that $\Phi(1)=1$ (for otherwise we may replace $\Phi$ by $\widehat{\Phi}=(\Phi(1))^{-1}\Phi)$; from Fact \ref{Fact:Orlicz} $L_\Phi[0,1]\simeq L_{\widehat{\Phi}}[0,1]$ and consequently $\Lip_0(L_\Phi[0,1])\simeq \Lip_0(L_{\widehat{\Phi}}[0,1])$).
Because $\Phi$ is convex, $\Phi(0)=0$ and $\Phi(1)=1$, it follows that $\Psi:[0,\infty)\to [0,\infty)$ given by $\Psi(t)=t^2\chi_{[0,1]}+(2\Phi(t)-1)\chi_{[1,\infty)}$ 
is a Young function equivalent to $t^2$ at $0$ and equivalent to $\Phi$ at $\infty$. Here we recall that two Young functions $\Phi,\Psi:[0,\infty)\to [0,\infty)$ are said to be equivalent at $\infty$ (resp. $0$) if there exists $0<t_0<\infty$ and $A,B,a,b>0$ such that $A\Psi(at)\leq \Phi(t)\leq B\Psi(bt)$ whenever $t\geq t_0$ (resp. $t \leq t_0$), see \cite[\S 7]{JoMaSchTz}. Since $\Phi$ is a $N$-function and $L_\Phi[0,1]$ is reflexive (see Remark \ref{Rem:aux}), by \cite[Theorem 8.6 and the second remark after its statement]{JoMaSchTz}, we have $L_\Phi[0,1]\simeq L_\Psi[0,\infty)$. Then, according to Theorem \ref{thm:OrliczLip}:
\[\Lip_0(L_\Phi[0,1])\simeq\Lip_0(L_\Psi[0,\infty))\simeq\bigoplus_{\ell_{\infty}}\Lip_0(\ell_{(1/n)\Psi}).\]
Since $(1/n)\Psi$ is equivalent to $t^2$ at $0$, from \cite[Proposition 4.a.5.]{LiTz77} we deduce that $\ell_{(1/n)\Psi}\simeq \ell_2$ for each $n \in \Nat$.
\end{proof}

It seems to be an interesting problem to find classes of Orlicz functions $\Phi$ for which we have $\Lip_0(\ell_\Phi)\simeq \Lip_0(L_\Phi)$ e.g. by observing that $(\ell_\Phi, L_\Phi)$ is a good pair (it seems that our proof would work only for functions $\Phi$ satisfying $\Phi(2x)=K\Phi(x)$).

\section{Questions}

In this section we show that our results open quite a lot of space for further investigations. We suggest several natural open problems which we consider interesting.
\subsection{Question related to finite-dimensional spaces}
In this paper, we have considered finitely generated and Lie groups as a natural class of pointed metric spaces for which it is interesting to investigate the corresponding spaces of Lipschitz functions and their preduals, Lipschitz-free spaces. It is of interest to demonstrate that `different' groups have non-isomorphic $\Lip_0$-spaces and Lipschitz-free spaces. One must be careful with the notion of being different and with the class of groups in consideration. For example, all non-abelian free groups are bi-Lipschitz equivalent to each other. Moreover, it follows from the result of Godard \cite{Godard} that for every countable free group $F$ we have $\F(F)\simeq \ell_1$ and  $\Lip_0(F)\simeq \ell_\infty$ since $F$ isometrically embeds into an $\Rea$-tree. This includes $F=\zet$ which is not even quasi-isometric to any non-abelian free group. So we face some restrictions. However, it leads us to pose the following question for hyperbolic groups which is motivated by the fact that asymptotic cones of finitely generated hyperbolic groups are (non-separable) $\Rea$-trees (see e.g. \cite[Proposition 11.167]{DrutuKapovich}.
\begin{question}
Is $\F(\Gamma)\simeq \ell_1$ (and $\Lip_0(\Gamma)\simeq \ell_\infty$) for every finitely generated hyperbolic group $\Gamma$?
\end{question}
Both positive and negative answer would be intriguing. Negative answer would show that the functor $\F(\cdot)$ can distinguish even some hyperbolic groups, while positive answer would imply that there exists a group $\Gamma$ with property (T) with $\F(\Gamma)\simeq \ell_1$ - this follows since there are hyperbolic groups with property (T) (see e.g. \cite[Chapter 19]{DrutuKapovich}). It is known that every action by affine isometries of a group with property (T) on $\ell_1$ with its standard norm has a fixed point (see \cite[Corollary D]{BGM}). On the other hand, since the action of $\Gamma$ on itself by left multiplication induces a proper action by affine isometries on $\F(\Gamma)$, it would imply that there is a renorming of $\ell_1$ on which $\Gamma$ has a proper action by isometries. Note that it is not clear which groups with property (T) have fixed points for isometric actions on all spaces isomorphic to a Hilbert space (see e.g. \cite{Koi} where this is discussed). In particular, we emphasize there is a conjecture of Shalom that there is a hyperbolic group with property (T) with a proper isometric action on a space isomorphic to a Hilbert space (see again \cite{Koi}). This last observation suggests that in general for a group $\Gamma$ with property (T), it is likely that $\F(\Gamma)\not\simeq \ell_1$, or at least the Banach-Mazur distance from $\ell_1$ should be `large'.
\medskip

On the other hand, one of the main open problems in the area of Lipschitz-free spaces and spaces of Lipschitz functions is to distinguish $\Lip_0(\Rea^m)$ and $\Lip_0(\Rea^n)$, or equivalently according to our results, $\Lip_0(\zet^m)$ and $\Lip_0(\zet^n)$, for $m\neq n$. The same for the spaces $\F(\Rea^m)$ and $\F(\Rea^n)$, or $\F(\zet^m)$ and $\F(\zet^n)$. This research project could be brought to a wider context of Carnot groups and their discrete counterparts, finitely generated nilpotent torsion-free group, as considered in this paper. 
\begin{question}
Are the functors $\Lip_0(\cdot)$ and $\F(\cdot)$ complete invariants of quasi-isometry and/or isomorphism for the classes of simply connected nilpotent Lie groups, resp. finitely generated torsion-free nilpotent groups? 
\end{question}
The only results that we are aware of are the following. Naor and Schechtman prove (see \cite{NaorSchecht}) that $\Lip_0(\Rea)\not\simeq \Lip_0(\Rea^2)$ and $\Lip_0(\zet)\not\simeq \Lip_0(\zet^2)$, and the same for the functor $\F(\cdot)$. It immediately follows by the main result from \cite{ChKl} that $\F(\zet)\not\simeq\F(H_3(\zet))$ and $\F(\Rea)\not\simeq\F(H_3(\Rea))$.

To put the previous question into a wider context, we note that it is still one of the main open problems in the metric geometry of Lie groups whether two quasi-isometric simply connected nilpotent Lie groups are isomorphic (see e.g. \cite[Problem 25.44]{DrutuKapovich}). Also, if $\Gamma$ and $\Delta$ are two quasi-isometric finitely generated torsion-free nilpotent groups, then their associated Carnot groups are isomorphic (see \cite{Pansu2} or \cite[Theorem 25.43]{DrutuKapovich}). On the other hand, it is still open whether their Mal'cev closures are isomorphic, we refer to \cite{Sauer} for a discussion about this last problem.
\subsection{Question related to infinite-dimensional spaces}
The result that $\Lip_0(\zet^d)\simeq\Lip_0(\Rea^d)$, which motivated our research, can be considered for generalization into two natural directions. Either to drop the assumption of commutativity and keep the assumption of finite-dimensionality, which is what we did when we considered the Carnot groups as finite-dimensional generalization of $\Rea^d$. Or the other way round, to keep the assumption of commutativity and drop the assumption of finite-dimensionality, which leads to the following question.
\begin{question}\label{q:netsInBanach}
Let $\Net_X$ be a net in a separable Banach space $X$. Is $\Lip_0(\Net_X)\simeq \Lip_0(X)$?
\end{question}

We do not know the answer to Question~\ref{q:netsInBanach} even for classical spaces such as $X = \ell_1$ or $X=c_0$. Note that if it was true that $\Lip_0(X)\simeq \Lip_0(\mathcal{N}_X)$ for every separable Banach space $X$, then uniformly homeomorphic Banach spaces would have isomorphic spaces of Lipschitz functions.

\begin{question}Let $\Net_X$ be a net in a separable Banach space $X$. Do we have $\Lip_0(\Net_X)\simeq \bigoplus_{\ell_\infty} \Lip_0(\Net_X)$?
\end{question}

We have been able to prove that $\Lip_0(\ell_p)\simeq \Lip_0(L_p)$ either using the `global geometric fact' that $(\ell_p,L_p)$ is a good pair (see Theorem~\ref{thm:OrliczLip}), or using the `local geometric fact' that locally $\ell_p$ and $L_p$ are the same (see Theorem~\ref{thm:scriptLp}). Both locally and globally, $L_p$ and $L_q$ for $1\leq p<q<\infty$ are geometrically quite different, therefore the following is a natural conjecture.
\begin{question}
Is $\Lip_0(L^p)\not\simeq\Lip_0(L^q)$, for $1\leq p<q<\infty$?
\end{question}
Another natural question concerns the preduals of $\Lip_0(L_p)$. For the good pairs of the form $(\Gamma, G_\infty)$, we can distinguish their free spaces, however these tools are not available for the pair $(\ell_p,L_p)$.

\begin{question}
Is $\F(\ell_p)\simeq \F(L_p)$ for every (some) $1\leq p\neq 2<\infty$?
\end{question}

\begin{question}\label{q:roleOfEll2}Let $X$ be a separable Banach space. Is it true that $\Lip_0(\ell_2)\complemented \Lip_0(X)$?
\end{question}

We do not know the answer to Question~\ref{q:roleOfEll2} even for $X=c_0$ and $X=\ell_1$.

\begin{question}\label{q:approximableByFinite}Is $\F(\ell_p)\simeq \bigoplus_{\ell_1} \F(\ell_p^n)$ for some (or even every) $1\leq p < \infty$?
\end{question}

Note that if the answer to Question~\ref{q:approximableByFinite} is positive for some $p$ then, using the argument from Remark~\ref{rem:freeIzo}, we would have $c_0\not \hookrightarrow \F(\ell_p)$, which would in the case of $p=1$ or $p=2$ answer a question raised already in \cite{CDW}.
\medskip

Our last question is related to our poor knowledge of $\F(\Rea^d)$ spaces.
Since we have proved that $\Lip_0(\Rea^d)\simeq\Lip_0(\zet^d)$, it seems to be natural to investigate the properties of Banach spaces $\Lip_0(\zet^d)$. By the already known results, it seems that those spaces share many properties of $\ell_\infty$ and so it is natural to wonder how far one can go with this. The first natural and interesting property could be the Grothendieck property. Note that every von Neumann algebra is a Grothendieck space, see \cite{P94} or \cite{PoPe10}, so a positive answer would support the intuition that Lipschitz spaces behave in many ways as von Neumann algebras.

\begin{question}\label{q:grot}Is $\Lip_0(\mathbb Z^d)$ a Grothendieck space for each $d\in\Nat$?
\end{question}

\noindent{\bf Acknowledgements.}
We would like to thank to Enrico Le Donne and Yves de Cornulier for responding our questions about Carnot groups and to Lubo\v{s} Pick and William B. Johnson for discussing with us the topic of Orlicz spaces. Moreover, we thank to the anonymous referee for many helpful remarks.

\end{document}